\documentclass[12pt]{amsart}

\usepackage{amssymb}

\usepackage{amsmath}

\newtheorem{theorem}{Theorem}[section]
\newtheorem{lemma}[theorem]{Lemma}

\newtheorem{fact}[theorem]{Fact}

\newtheorem{proposition}[theorem]{Proposition}
\newtheorem{corollary}[theorem]{Corollary}
\newtheorem{claim}[theorem]{Claim}
\newtheorem{subclaim}{Subclaim}[theorem]

\theoremstyle{definition}
\newtheorem{definition}[theorem]{Definition}
\newtheorem{notation}[theorem]{Notation}

\let \restr = \upharpoonright

\let \into = \longrightarrow

\let \tld = \tilde
\let \wdtld = \widetilde

\let \sub = \subseteq
\let \elsub = \preccurlyeq

\let \av = \arrowvert
\let \ov = \overline
\let \a = \alpha
\let \b = \beta
\let \g = \gamma
\let \d = \delta
\let \e = \epsilon

\let \l = \lambda
\let \k = \kappa
\let \m = \mu
\let \n = \nu
\let \p = \pi

\let \t = \theta

\let \D = \Delta

\let \s = \sigma
\let \x = \xi
\let \z= \zeta
\let \o = \omega

\let \S = \Sigma

\let \al = \aleph
\let \la = \langle
\let \ra = \rangle
\let \mtcl = \mathcal
\let \mtbb = \mathbb
\let \it = \item

\DeclareMathOperator{\dom}{dom}
\DeclareMathOperator{\range}{range}
\DeclareMathOperator{\cf}{cf}

\DeclareMathOperator{\rank}{rank}
\DeclareMathOperator{\supp}{supp}
\DeclareMathOperator{\Lim}{Lim}
\DeclareMathOperator{\V}{\mathbf V}
\DeclareMathOperator{\W}{\mathbf W}
\DeclareMathOperator{\ZFC}{ZFC}
\DeclareMathOperator{\MA}{\textsf{MA}}
\DeclareMathOperator{\TC}{\textsf{TC}}

\title{A generalization of Martin's Axiom}

\author[D. Asper\'o]{David Asper\'o}

\thanks{Asper\'{o} was partially supported by the Austrian Science Fund (FWF) Project M 1408-N25. Mota was supported by the FWF Project P22430. Both authors were also partially supported by Ministerio de
Educaci\'{o}n y Ciencia Project MTM2008--03389 (Spain) and by Generalitat de Catalunya Project 2009SGR--00187 (Cata\-lonia).}

\address{David Asper\'o, School of Mathematics, University of East Anglia, Norwich NR4 7TJ, UK}

\email{d.aspero@uea.ac.uk}

\author[M.A. Mota]{Miguel Angel Mota}

\address{Miguel Angel Mota,  Department of Mathematics,
University of Toronto,
Toronto, Ontario,
CANADA
M5S 2E4}

\email{motagaytan@gmail.com}

\begin{document}


\date{}
\maketitle
\pagestyle{myheadings}
\markright{A generalization of Martin's Axiom}

\begin{abstract}
We define the $\al_{1.5}$--chain condition. The corresponding forcing axiom is a generalization of Martin's Axiom; in fact, $\MA^{1.5}_{<\kappa}$ implies $\MA_{<\k}$. Also,  $\MA^{1.5}_{<\k}$ implies
certain uniform failures of club--guessing on $\o_1$ that do not seem to have been considered in the literature before. We show, assuming $\textsf{CH}$ and given any regular cardinal $\k\geq\o_2$ such that $\m^{\al_0}< \kappa$ for all $\m < \k$ and such that $\diamondsuit(\{\a<\k\,:\,\cf(\a)\geq\o_1\})$ holds, that there is a proper $\al_2$--c.c.\ partial order of size $\k$ forcing $2^{\al_0}=\k$ together with $\MA^{1.5}_{<\k}$.  \end{abstract}

\section{A generalization of Martin's Axiom. And some of its applications.}

Martin's Axiom, often denoted by $\MA$, is the following very well--known and very classical forcing axiom: If $\mtbb P$ is a partial order (poset, for short) with the countable chain condition\footnote{A partial order has the countable chain condition (is c.c.c.) if and only if it has no uncountable antichains. More generally, given a cardinal $\k$, a partial order has the $\k$--chain condition (is $\k$--c.c.) if it has no antichains of size $\k$.} and $\mtcl D$ is a collection of size less than $2^{\al_0}$ consisting of dense subsets of $\mtbb P$, then there is a filter $G\sub\mtbb P$ such that $G\cap D\neq\emptyset$ for every $D\in\mtcl D$.

Martin's Axiom is obviously a weakening of the Continuum Hypothesis. Given a cardinal $\l$, $\MA_{\l}$ is obtained from considering, in the above formulation of $\MA$, collections $\mtcl D$ of size at most $\l$ rather than of size less than $2^{\al_0}$. Martin's Axiom becomes interesting when $2^{\al_0}>\al_1$.

$\MA_{\aleph_1}$ was the first forcing axiom ever considered (\cite{Martin-Solovay}). As observed by  D.\ Martin, the consistency of $\MA$ together with $2^{\al_0}>\al_1$ follows from generalizing the Solovay--Tennenbaum construction of a model of Suslin's Hypothesis by iterated forcing using finite supports (\cite{Solovay-Tennenbaum}).
Since then, a plethora of applications of $\MA$ (+ $2^{\al_0}> \al_1$) have been discovered in set theory, topology, measure theory, group theory, and other areas (\cite{Fremlin} is a classical reference).

In the present paper we generalize Martin's Axiom to the class of posets satisfying what we call here the $\al_{1.5}$--\emph{chain condition} ($\al_{1.5}$--c.c.). This terminology is intended to highlight the fact that every poset with the countable chain condition (i.e., with the $\al_1$--c.c.) is in our class and that every poset in our class has the $\al_2$--c.c.
Of course it follows from the first inclusion that for every cardinal $\l$, the forcing axiom $\MA^{1.5}_\l$ for the class of all $\al_{1.5}$--c.c.\ posets relative to collections of $\l$-many dense sets implies $\MA_\l$. Furthermore, as to the consistency of $\MA^{1.5}_\lambda$, there is no restriction on $\l$ other than $\l < 2^{\al_0}$. More precisely, the same construction shows that $\MA^{1.5}_{\al_1}$, $\MA^{1.5}_{\al_{727}}$, $\MA^{1.5}_{\al_{\o_2 + \o + 3}}$, and so on are all consistent.\footnote{The same is true for the Solovay--Tenennbaum construction, i.e., the same construction shows the consistency of Martin's Axiom together with $2^{\al_0}$ being $\al_2$, $\al_{728}$, $\al_{\o_2 + \o + 4}$, and so on.} This construction will take the form of a forcing iteration, in a broad sense of the expression, involving certain partial symmetric systems of countable submodels as side conditions (cf.\ our earlier work in \cite{ASP}).

Note that the collapse of $\o_1$ to $\o$ with finite conditions has size $\al_1$ and therefore has the  $\al_2$--chain condition. The forcing axiom for collections of $\al_1$--many dense subsets of this poset is obviously false. This shows of course that some restriction is necessary in order to obtain a consistent forcing axiom for posets with the $\al_2$--c.c., even relative to collections of $\al_1$--many dense sets. On the other hand, the definition of our class of posets will be wide enough to contain all c.c.c.\ posets and in fact to make the corresponding forcing axiom $\MA^{1.5}_\l$ (for $\l \geq \aleph_1$) strictly stronger than $\MA_\l$. In fact, we will show for example that $\MA^{1.5}_\l$ implies certain `uniform' failures of Club Guessing on $\o_1$ that do not seem to have been considered before in the literature, and which do not follow from $\MA_\l$. As a matter of fact we do not know how to show the consistency of these statements by any method other than ours.

Before going on we will fix a setting for our notation. For the most part we will follow set--theoretic standards as set forth for example in \cite{JECH} and in \cite{KUNEN}. In particular, if $\la\mtcl P_\a\,:\,\a\leq\k\ra$ is a forcing iteration and $\a\leq\k$, $\dot G_\a$ denotes the canonical $\mtcl P_\a$--name for the generic filter added by $\mtcl P_\a$, and $\leq_\a$ typically denotes the extension relation on $\mtcl P_\a$. We will also make use of less standard pieces of notation. For example, if $N$ is a set for which $N\cap \o_1$ is an ordinal, we will use $\d_N$ to denote this intersection and will call $\d_N$ the {\emph height of $N$}.
Other pieces of notation will be introduced when needed.

Let $\mtbb P$ be a poset and let $N$ be a submodel of $\langle H(\theta), \in \rangle$ such that $\mtbb P \in N$ and $\t \geq \av \TC(\mtbb P)\av^+$.\footnote{Where $\TC(x)$ denotes the transitive closure of $x$.} Recall that a $\mtbb P$--condition $p$ is $(N, \mtbb P)$--generic if for every extension $p'$ of $p$ and every dense subset $D$ of $\mtbb P$ belonging to $N$ (equivalently, every maximal antichain $D$ of $\mtbb P$ belonging to $N$) there is some condition in $D \cap N$ compatible with $p'$. Also, $\mtbb P$ is proper (\cite{Shelah}) if for every cardinal $\t \geq \av \TC(\mtbb P)\av^+$ it holds, for club--many countable $N\elsub H(\t)$, that for every $p\in N\cap\mtbb P$ there is a condition $q$ in $\mtbb P$ extending $p$ and such that $q$ is $(N, \mtbb P)$--generic. Every poset $\mtbb P$ with the countable chain condition is proper as every condition is $(N, \mtbb P)$--generic for every $N$ as above.

Now we may proceed to the definition of the $\al_{1.5}$--c.c.

\begin{definition}\label{aleph_{1.5}--c.c.}
Given a poset $\mtbb P$, we will say that \emph{$\mtbb P$ has the $\al_{1.5}$--chain condition} ($\mathbb P$ is \emph{$\al_{1.5}$--c.c.} for short) if and only if for every regular cardinal $\lambda\geq\ |\TC(\mtbb P)|^+$ there is a club $D\sub[H(\l)]^{\al_0}$ such that for every finite $\{N_i\,:\, i < n\}\sub D$ and every $p\in\mtbb P$, if $p\in N_j$ for some $j$ such that $\d_{N_j}=\min\{\d_{N_i}\,:\,i<n\}$, then there is some condition extending $p$ and $(N_i,\, \mtbb P)$--generic for all $i$.\footnote{Note that we are not assuming that $p \in \bigcap\{N_i\, :\, i < m\}$.}
\end{definition}

Since the intersection of two clubs is again a club, we may assume that all members of $D$ in the above definition are countable elementary substructures of $\langle H(\lambda), \in \rangle$. Also note that a poset $\mtbb P$ has the $\al_{1.5}$--chain condition if and only if it satisfies the above definition for some $\l \geq\ |\TC(\mtbb P)|^+$ (this is true because the projection of a club is again a club). Finally note that every c.c.c.\ poset is $\al_{1.5}$--c.c.\ and that every $\aleph_{1.5}$--c.c.\ poset is proper.

\begin{proposition}
If $\mtbb P$ has the $\al_{1.5}$--c.c., then $\mtbb P$ has the $\al_{2}$--c.c.
\end{proposition}

\begin{proof}
Let $A$ be a maximal antichain of $\mtbb P$ such that $|A| \geq \aleph_2$. Let $\l$ and $D$ be as in Definition \ref{aleph_{1.5}--c.c.} and let $N_p$ be, for every $p\in A$, a member of $D$ such that  $A$, $p\in N_p$. Let $\d<\o_1$ be such that the set $A'$ of $p\in A$ such that $\d_{N_p}=\d$ is uncountable. Pick some $p\in A'$ and some $p'\in A'\setminus N_p$. Then, since $\d_{N_{p'}}=\d_{N_p}=\d$ and $p'\in N_{p'}$, there is a condition $q$ extending $p'$ such that $q$ is $(N_p, \mtbb P)$--generic. In particular $q$ is compatible with some $\bar p\in A\cap N_p$. This is a contradiction since $A$ is an antichain and $\bar p\neq p'$.
\end{proof}

It is clear that the intersection of any c.c.c.\ poset $\mtcl Q $ with any $N\elsub H(\t)$ such that $\mtcl Q\in N$ (for any $\t>\av \TC(\mtcl Q)\av$) is itself c.c.c. In fact, for every cardinal $\k$ the $\k$--c.c.\ is hereditary, in the sense that if $\mtcl Q$ has the $\k$--c.c., $\t$ is a cardinal such that $\av \TC(\mtcl Q)\av<\t$, and $N\elsub H(\t)$ is such that $\mtcl Q\in N$, then $\mtcl Q \cap N$ is $\k$--c.c. The reason is simply that two condition in $\mtcl Q\cap N$ are compatible in $\mtcl Q$ if and only if they have a common extension in $\mtcl Q\cap N$.  This fact, for $\k=\al_1$, is a crucial ingredient in the Solovay--Tenenbaum consistency proof of $\MA$ with $2^{\al_0}= \l$, where one argues that if $\mtcl Q$ is a c.c.c.\ poset in the final extension $\W$ and $(A_i\mid i<\m)\in \W$ is a sequence of maximal antichains of $\mtcl Q$ for $\m<\l$, then there is a c.c.c.\ poset $\mtcl R\sub\mtcl Q$ such that each $A_i$ is a maximal antichain of $\mtcl R$ and such that moreover $\av\mtcl R\av<\l$.\footnote{One then argues that there is some stage of the iteration at which both $\mtcl R$ and $(A_i\mid i<\m)$ were defined and at which we forced with $\mtcl R$.}   It is not clear whether or not the $\al_{1.5}$--c.c.\ is hereditary in the above sense, but in any case we can prove the following weak form of this property, which will suffice for our consistency proof (Theorem \ref{mainthm}).

\begin{fact}\label{factproper2}
Suppose $\mtcl Q$ has the $\al_{1.5}$--c.c., $\t$ is a cardinal such that there is some strong limit cardinal $\chi$ with $\t>\chi>\av\TC(\mtcl Q)\av$,  and $M$ is an elementary substructure of $H(\t)$ such that $\mtcl Q\in M$ and $^{\o_1}M\sub M$. Then $\mtcl R = \mtcl Q\cap M$ is a complete suborder of $\mtcl Q$ with the $\al_{1.5}$--c.c.
\end{fact}

\begin{proof}
To start with, notice that $\mtcl R$ has the $\al_2$--chain condition and is a complete suborder of $\mtcl Q$ since $M$ is closed under $\o_1$--sequences.
To see that $\mtcl R$ has the $\al_{1.5}$--c.c., let $\chi\geq\av \TC(\mtcl Q)\av^+$ be a strong limit cardinal in $M$ and $W$ a bijection between $\chi$ and $H(\chi)$ also living in $M$. Let $D\sub[H(\chi)]^{\al_0}$ be a club witnessing the $\al_{1.5}$--c.c.\ of $\mtcl Q$, let  $p\in \mtcl R$, and let $N_0,\ldots, N_m$ be a finite set of countable elementary substructures of $\langle H(\chi), \in, W\rangle$ in $D$ containing $\mtcl Q$ and such that $p\in N_j$ for some $j$ with $\d_{N_j}$ minimal.

\begin{subclaim}\label{copias}
In $M$ there are countable elementary substructures $M_0, \ldots, M_m$ of  $(H(\chi), \in, W)$ such that for all $i$,

\begin{itemize}
\it[(a)] $M_i \in M$,

\it[(b)] $M_i\cap N_i=N_i\cap M$,\footnote{In particular, $\mtcl Q\in M_i$.}, and

\it[(c)] there is an isomorphism $\varphi_i: (N_i, \in, W)\into (M_i, \in, W)$ fixing $N_i\cap M_i$.

\end{itemize}
\end{subclaim}

\begin{proof}

By elementarity and since $M$ contains all the relevant parameters (all the real numbers, $M \cap N_i$ and $W$), $M$ also contains a set $\{M_0,\ldots, M_m\}$ such that for each $i$,  \begin{itemize}
\it $M \cap N_i\sub M_i$ and

\it there is
an isomorphism $\varphi_i$ between $(N_i, \in, W, M_i\cap N_i\cap\chi)$ and $(M_i, \in, W, M_i\cap M_i\cap \chi)$.
\end{itemize}

Note that $\varphi_i$ has to be the identity on $M_i \cap N_i \cap \chi$. Finally, since there is  a bijection $W: \chi \into H(\chi)$ definable in $(H(\chi), \in, W)$, we have that $\varphi_i$ fixes $M \cap N_i \cap \chi$ if and only if it fixes $M \cap N_i$.

\end{proof}

Fix $M_0, \ldots, M_m$ as in the above claim. Since $p\in M_j$ and $\d_{M_j}=\min\{\d_{M_0},\ldots, \d_{M_m}\}$, there is  a condition $\ov p\in\mtcl R$ such that $\ov p\leq_{\mtcl Q}p$ and such that $\ov p$ is $(M_i, \mtcl Q)$--generic for every $i$.  Suppose towards a contradiction that there is $p'\leq_{\mtcl R}\ov p$, $i_0< m+1$, and some maximal antichain $A$ of $\mtcl R$ in $N_{i_0}$ such that no condition in $A\cap N_{i_0}$ is compatible with $p'$.
Let $p^\ast \in\mtcl R$ be a common extension of $p'$ and of some $p''\in \varphi_i(A)\cap M_{i_0}$. To see that there are such $p^\ast$ and $p''$, note that $\varphi_{i_0}(A)\in M_i$ is a maximal antichain of $\mtcl Q$.  This is true since $\mtcl Q\in N_{i_0}\cap M_{i_0}$ and, therefore, $\varphi_{i_0}$ fixes $\mtcl Q$.

Now note that $A\in M$ since $M$ is closed under $\o_1$--sequences and $\av A\av\leq\al_1$. It follows that $\varphi_{i_0}(A)=A$ since $\varphi_{i_0}$ is the identity on $M\cap N_{i_0}$.  Also, $p''\in N_{i_0}$. To see this, take a surjection $f:\o_1\into A$ in $N_{i_0}\cap M$ (take for example the $W$--first surjection $f:\o_1\into A$). Then $\varphi_{i_0}(f) = f\in M_{i_0}$ is a surjection from $\o_1$ onto $\varphi_{i_0}(A)=A$. Let $\z\in M_{i_0} \cap\o_1$ such that $f(\z)=p''$. Then $\z\in N_{i_0}\cap\o_1$ and so $p''= f(\z)\in N_{i_0}$. This contradiction finishes the proof. \end{proof}

\begin{corollary}\label{cor5}
If $\mtcl Q$ has the $\al_{1.5}$--chain condition and $X$ is a subset of $\mtcl Q$, then  there is $\mtcl R$ such that

\begin{itemize}

\it[(i)] $\mtcl R$ is a complete suborder of $\mtcl Q$,

\it[(ii)] $\mtcl R$ has the $\al_{1.5}$--chain condition,

\it[(iii)] $X\sub\mtcl R$, and

\it[(iv)] $\av\mtcl R\av \leq
 \av X\av^{\al_1}$

\end{itemize}
\end{corollary}

We are ready now to define our generalisation of $\MA$.

\begin{definition}\label{def00}
Given a cardinal $\l$, let $\MA^{1.5}_{\l}$ be the following statement: For every $\al_{1.5}$--c.c.\ poset $\mtbb P$ and every collection $\mtcl D$ of size $\l$ consisting of dense subsets of $\mtbb P$ there is a filter $G\sub\mtbb P$ such that $G\cap D\neq\emptyset$ for all $D\in \mtcl D$.
\end{definition}

In the remainder of this section we present more consequences of $\MA^{1.5}_\l$.

\begin{proposition}\label{productiveness}
$\MA^{1.5}_{\al_2}$ implies that if $\mtbb P$ is $\al_{1.5}$--c.c.\ and $X\in[\mtbb P]^{\al_2}$, then there is $Y\in[X]^{\al_2}$ such that every nonempty $\s\in[Y]^{{<}\o}$ has a lower bound in $\mtbb P$. In particular, $\MA^{1.5}_{\al_2}$ implies that any finite support product of $\al_{1.5}$--c.c.\ posets has the  $\al_2$--c.c.
\end{proposition}

Note that there are $\textsf{ZFC}$ examples showing that the $\al_2$--c.c.\ is not productive (\cite{Shelah-colouring-non-productivity}).
The proof of Proposition \ref{productiveness} is essentially the same as the classical proof that the c.c.c.\ is productive under $\MA_{\aleph_1}$ (see for example \cite{KUNEN}, Lemma 2.23 and Theorem 2.24). For the reader's convenience we include the argument here.

Say that a topological space is $\al_{1.5}$--c.c.\ if its topology, ordered by inclusion, is an $\al_{1.5}$--c.c.\ partial order. Of course,  a partial order $\mtbb P$ is $\al_{1.5}$--c.c.\ if and only if the topology on $(\mtbb P, \leq)$ generated by the collection of cones $\{p\in\mtbb P\,:\, p\leq q\}$ (for $q\in\mtbb P$) is $\al_{1.5}$--c.c. Hence, in order to prove the first assertion of Proposition \ref{productiveness} it suffices to prove the following proposition.  The second assertion follows then immediately from the first by a standard $\D$--system argument.

\begin{proposition}\label{productiveness2}
$\MA^{1.5}_{\al_2}$ implies that if $X$ is an $\al_{1.5}$--c.c.\ topological space and $\{U_\a\,:\,\a<\o_2\}$ is a collection of nonempty open subsets of $X$, then there is $I\in[\o_2]^{\al_2}$ such that $\{U_\a\,:\,\a\in I\}$ has the finite intersection property.\end{proposition}

\begin{proof}
Let $(V_\a)_{\a<\o_2}$ be the decreasing sequence of open sets given by $V_\a=\bigcup_{\b>\a}U_\b$. Note that there must be $\a<\o_2$ such that $\bar V_\a=\bar V_\b$ for all $\b>\a$ (where $\bar V$ denotes topological closure). Otherwise there would be $Z\sub \o_2$ unbounded in $\o_2$ such that $V_\a\setminus \bar V_\b\neq\emptyset$ for all $\a<\b$ in $Z$, which would yield the existence of $\al_2$--many pairwise disjoint open subsets of $X$, a contradiction since $X$ is $\al_{1.5}$--c.c.

Let now $\a$ be as above, and let $\mtbb P$ be the partial order of all nonempty open subsets of $V_\a$ ordered by inclusion. For every $\b>\a$ let $D_\b$ be the set of $U\in \mtbb P$ such that $U\sub U_\g$ for some $\g>\b$. Since $\bar V_\a\sub \bar V_\b$, it follows that $D_\b$ is a dense subset of $\mtbb P$. Hence, since $\mtbb P$ is $\al_{1.5}$--c.c., by $\MA^{1.5}_{\al_2}$ we may find a filter $G\sub\mtbb P$ meeting $D_\b$ for all $\b>\a$, which gives the desired conclusion.
\end{proof}

Note that the above (weak) productiveness result, as well as the productiveness of c.c.c.\ under $\textsf{MA}_{\al_1}$, are instances of the following general result, which is proved exactly as above.

\begin{proposition} If $\Gamma$ is a class of posets with the $\k$--c.c., then the forcing axiom for $\Gamma$ and for collections of $\k$--many dense sets implies that for every $\mtbb P\in\Gamma$ and every $X\in[\mtbb P]^\k$ there is some $Y\in[X]^\k$ such that every nonempty $\s\in[Y]^{{<}\o}$ has a lower bound in $\mtbb P$; in particular, this forcing axiom implies that any finite support product of members of $\Gamma$ has the $\k$--c.c.
\end{proposition}

Next we will show that $\MA^{1.5}_\l$ implies certain uniform failures of Club Guessing on $\o_1$, as advertised in the abstract. It will be convenient to consider the following natural notion of rank (see for example \cite{ASP}).

\begin{definition}
Given a set $X$ and an ordinal $\d$, we define the \textit{Cantor--Bendixson rank of $\d$ with respect to $X$},
$\rank(X, \d)$, by specifying that
\begin{itemize}

\it $\rank(X, \d) \geq 1$ if and only if $\d$ is a limit point of ordinals in $X$, and that

\it if $\m >1$, $\rank(X, \d) \geq \m$ if and only and for every $\eta < \mu$, $\d$ is a limit of ordinals $\e$ with $\rank(X, \e)\geq\eta$.

\end{itemize}

\end{definition}

Given an ordinal $\tau$, we will say that a set $X$ of ordinals is \emph{$\tau$--thin} in case $\rank(X, \d)\leq\tau$ for all ordinals $\d$.

\begin{definition} Given ordinals $\tau$ and $\l$, $\tau<\omega_1$, $(\cdot)^\tau_\lambda$ is the following statement: For every sequence $(A_i)_{i<\l}$, if $A_i$ is a $\tau$--thin subset of $\omega_1$  for all $i<\lambda$, then there is a club $C\sub\omega_1$ such that $\av C\cap X_i\av < \aleph_0$ for all $i$.
\end{definition}

Clearly, if $\tau\leq\tau'$ and $\lambda\leq\lambda'$, then $(\cdot)^{\tau'}_{\lambda'}$ implies $(\cdot)^\tau_\lambda$. In particular, for every infinite $\tau$ and every $\lambda\geq \omega_1$, $(\cdot)^\tau_\lambda$ implies $\lnot\textsf{WCG}$, where $\textsf{WCG}$ denotes Weak Club Guessing on $\omega_1$ (cf.\ \cite{Shelah}).

\begin{fact} For every cardinal $\lambda\geq\omega_1$, the following weakening of $(\cdot)^\omega_\l$ implies $2^{\aleph_0}>\l$: For every sequence $(A^i_\delta)_{i<\l,\,\delta\in \Lim(\omega_1)}$, if each $A^i_\delta$ is a cofinal subset of $\delta$ of order type $\omega$, then there is a club $C\sub\omega_1$ such that $A^i_\delta \nsubseteq C$ for all $i<\lambda$ and $\delta\in \Lim(\omega_1)$.\end{fact}

\begin{proof} Suppose $2^{\al_0}\leq \l$ and let $(A^i_\delta)_{i < \l,\,\delta\in \Lim(\omega_1)}$ be such that for each $\delta$, $\{A^i_\delta\,:\,i<\lambda\}$ contains all cofinal subsets of $\delta$ of order type $\omega$. If $C\sub\omega_1$ is a club and $\delta\in C$ is a limit point of $C$, then there is $i<\lambda$ such that $A^i_\delta\sub C$.\end{proof}

Furthermore, it is not hard to see that full $(\cdot)^\omega_\lambda$ implies $\mathfrak b > \lambda$.

\begin{fact}\label{unifnoclub-guessing}
For every cardinal $\l\geq\o_1$, $\MA^{1.5}_\l$ implies $(\cdot)^\tau_\l$ for every $\tau<\omega_1$.\end{fact}

\begin{proof}
Let $(A_i)_{i<\l}$ be as in the definition of $(\cdot)^\tau_\l$. Let $\mtbb P$ consist of all pairs $(f, X)$ such that

\begin{itemize}

\it[(a)] $f \subseteq \o_1 \times \o_1$ is a finite function such that $\rank(f(\n), f(\n))\geq\n$ for every $\n\in \dom(f)$,

\it[(b)]  $X$ is finite set of triples $(i, \n, a)$ such that $i < \l$,  $\n\in \dom(f)$, $\rank(A_i, f(\n))<f(\n)$, and $a$ is a finite subset of  $f(\n)$, and

\it[(c)] for every $(i, \n, a)\in X$, $\range(f\restr  \n)\cap A_i=a$.
\end{itemize}

Given $\mtbb P$--conditions $(f_0, X_0)$ and $(f_1, X_1)$, $(f_1, X_1)$ extends $(f_0, X_0)$ if $f_0\sub f_1$ and $X_0 \sub X_1$.

Using the fact that there are only $\al_1$--many finite functions $f \subseteq \o_1 \times \o_1$, it is easy to check that $\mtbb P$ is $\al_{1.5}$--c.c.\ (for example by arguments as in Section 5 of \cite{ASP} for similar forcings).
Also, there is a collection $\mtcl D$ of $\max\{\l, \o_1\}$--many dense subsets of $\o_1$ such that if $G$ is a filter of $\mtbb P$ meeting all members of $\mtcl D$, then $\range(\bigcup\{f\,:\,(f, X)\in G\textrm{ for some }X\})$ is a club witnessing $(\cdot)^+_\l$ for $(A_i)_{i<\l}$.
\end{proof}

In contrast with Fact \ref{unifnoclub-guessing}, no forcing axiom $\MA_\l$  implies $(\cdot)^\tau_{\l}$ for any infinite $\tau<\o_1$ and any $\l \geq \o_1$. The reason is simply that $\MA_\l$ can always be forced by a c.c.c.\ forcing and c.c.c.\ forcing preserves $\textsf{WCG}$.

It seems that Theorem \ref{mainthm} below provides the first known construction of a model of $\MA_\l$, for any $\l>\o_1$, in which $\textsf{WCG}$ fails. Indeed, any long enough c.c.c.\ iteration $\mtcl P$ with finite supports producing a model of $\MA$ will necessarily force $\textsf{WCG}$. The reason is that $\mtcl P$ will add a Cohen real at some stage by its being a c.c.c.\ iteration with finite supports (and in fact at the $\o$--th stage). This Cohen real will add a weak club--guessing sequence $\vec C$ at that stage (for example by results in \cite{JUHASZ}) and, letting $\W$ be the corresponding model, $\vec C$ will remain weak club--guessing in the end since the tail forcing has the c.c.c.\ in $\W$ and therefore every club of $\o_1$ in the final model contains a club from $\W$.

\begin{definition} (\cite{ASP})
Given a partial order $\mtbb P$, \emph{$\mtbb P$ is finitely proper} if for every cardinal $\t\geq\av \mtbb P\av^+$, every finite set $\{N_i\,:\, i \in m\}$ of countable elementary substructures of $H(\t)$ containing $\mtbb P$, and every condition $p\in \bigcap\{N_i\, :\, i < m\} \cap\mtbb P$ there is some $q\in\mtbb P$ extending $p$ and $(N_i, \mtbb P)$--generic for all $i$.
\end{definition}

\begin{definition} (\cite{ASP}, essentially) Given a cardinal $\l$, $\textsf{PFA}^{\mbox{fin}}(\o_1)_\l$ is the forcing axiom for the class of finitely proper posets of size $\al_1$ and for collections of $\l$--many dense sets.\end{definition}

\begin{fact}
For every cardinal $\l\geq\o_1$, $\textsf{PFA}^{\mbox{fin}}(\o_1)_\l$ implies the following.

\begin{itemize}
\it[(1)] $\MA_{\o_1}$

\it[(2)] $(\cdot)^\tau_{\o_1}$ for every $\tau<\omega_1$.

\it[(3)] For every set $\mtcl F$ of size $\l$ consisting of functions from $\o_1$ into $\o_1$ there is a function $g:\o_1\into \o_1$ such that both $\{\n<\o_1\,:\,f(\n)<g(\n)\}$ and $\{\n<\o_1\,:\,f(\n)=g(\n)\}$ are unbounded for every $f\in\mtcl F$. In particular, $\mathfrak d(\o_1)>\l$, where $\mathfrak d(\o_1)$ is the dominating number for $^{\o_1}\o_1$ and where domination is understood relative to the ideal of countable sets.

\it[(4)] For every $\mathcal R\sub [\o_1]^{\al_1}$ of size $\l$ there is some $b\in[\o_1]^{\al_1}$ such that $\av a\cap b\av=\al_1$ and $\av a\setminus b\av=\al_1$ for every $a\in\mtcl R$. In particular, $\mathfrak r(\o_1)>\l$, where $\mathfrak r(\o_1)$ is the reaping number for $[\o_1]^{\al_1}$ and where the reaping relation is understood again relative to the ideal of countable sets.
\end{itemize}
\end{fact}

\begin{proof}
The proofs of (1) and (2) are either immediate or as the corresponding proofs from $\MA^{1.5}_\l$. (3) and (4) follow from considering Baumgartner's forcing for adding a club $C\sub\o_1$ by finite approximations.
\end{proof}

The following result is straightforward.

\begin{fact}\label{fin_proper} Every finitely proper poset of size $\al_1$ has the $\al_{1.5}$--chain condition. In particular, for every cardinal $\l$, $\MA^{1.5}_\l$ implies $\textsf{PFA}^{\mbox{fin}}(\o_1)_\l$.\end{fact}

 The proof of the main theorem in \cite{ASP} essentially shows the consistency of $\textsf{PFA}^{\mbox{fin}}(\o_1)_\l$ for arbitrary $\l$.

\section{The consistency of $\MA^{1.5}_\l$}\label{the_forcing_construction}

Our main theorem is the following:

\begin{theorem}\label{mainthm} ($\textsf{CH}$) Let $\k\geq\o_2$ be a regular cardinal such that $\m^{\al_0}< \kappa$ for all $\m < \k$ and $\diamondsuit(\{\a<\k\,:\,\cf(\a)\geq\o_1\})$ holds. Then there is a proper forcing notion $\mtcl P$ of size $\k$ with the $\al_2$--chain condition such that the following statements hold in the generic extension by $\mtcl P$:

\begin{itemize}

\it[(1)] $2^{\al_0}=\k$

\it[(2)] $\MA^{1.5}_{<2^{\al_0}}$
\end{itemize}
\end{theorem}

The proof of Theorem \ref{mainthm} is an elaboration of the proof of the main theorem in \cite{ASP}.  Our approach in that paper consisted in building a certain type of finite support forcing iteration $\la\mtcl P_\a\,:\,\a\leq\k\ra$ of length $\k$ (in a broad sense of `forcing iteration')\footnote{In the sense that $\mtcl P_\b$ is a regular extension of $\mtcl P_\a$ whenever $\a<\b\leq\k$. It follows of course that $\la\mtcl P_\a\,:\,\a\leq\k\ra$ is forcing--equivalent to a forcing iteration $\la\mtbb P_\a\,:\,\a\leq\k\ra$ in the ordinary sense (that is, such that $\mtbb P_{\a+1}\cong\mtbb P_\a\ast\dot{\mtbb Q}_\a$ for all $\a$, where $\dot{\mtbb Q}_\a$ is a $\mtbb P_\a$--name for a poset), but such a presentation of $\la\mtcl P_\a\,:\,\a\leq\k\ra$ is not really natural.} using what one may describe as finite ``symmetric'' systems of countable elementary substructures of a fixed $H(\k)$\footnote{This $\k$ is exactly the value that $2^{\al_0}$ attains at the end of the construction.} as side conditions. These systems of structures were added at the first stage $\mtcl P_0$ of the iteration.
Roughly speaking, the fact that the supports of the conditions in the iteration were finite ensured that the inductive proofs of the relevant facts -- mainly the $\al_2$--c.c.\ of all $\mtcl P_\a$ and their properness -- went through. The use of the sets of structures as side conditions was crucial in the proof of properness.\footnote{For more on the motivation of this type of construction see \cite{ASP}.}

Here we change the proof from \cite{ASP} in basically three ways. One of the changes is in the proof of the $\al_2$--chain condition of the iteration (Lemma \ref{cc}). The induction now is a bit different from the one in the corresponding proof from \cite{ASP} (and it has to be as the iterands we are now considering are no longer of size $\al_1$). In fact, our forcing has the $\al_2$--c.c.\ but it is not $\al_2$--Knaster (on the other hand, the forcing in \cite{ASP} is $\al_2$--Knaster). Another, more crucial, difference is the presence of a diamond--sequence which ensures that all posets with the $\al_{1.5}$--c.c.\ (with no restrictions on their size)
occurring in the final extension have been dealt with at $\k$--many stages during the iteration.
A third difference between the present construction and the construction from \cite{ASP} is the fact that our systems of structures now exhibit a weaker form of symmetry than the ones in \cite{ASP}; we call the systems in the present paper \textit{partial symmetric systems}, as opposed to \textit{symmetric systems} as in \cite{ASP} (see Subsection \ref{partial-symm-systems} for more on the difference between these two forms of symmetry). The need for the weaker form of symmetry comes from the part of the argument using a diamond--sequence where we infer that all $\al_{1.5}$--c.c.\ posets occurring in the final model are `captured' along the iteration.
It is perhaps worth pointing out that the proof of properness now is almost identical to the corresponding proof from \cite{ASP} (modulo notational changes).

Theorem \ref{mainthm} shows that all forcing axioms of the form $\MA^{1.5}_{<\k}$, for a fixed reasonably defined cardinal $\k$, are consistent (relative to the consistency of $\ZFC$).
This theorem shows also that no axiom of the form $\MA^{1.5}_\l$ decides the size of the continuum and thus, by Fact \ref{fin_proper}, fits nicely within the ongoing project of showing whether or not weak fragments of $\textsf{BPFA}$ imply $2^{\al_0} = \al_2$.\footnote{$\textsf{BPFA}$ is the assertion that $H(\o_2)^{\V}$ is a $\S_1$--elementary substructure of $H(\o_2)^{\V^{\mtbb P}}$ for every proper poset $\mtbb P$ (\cite{Bagaria}). J.\ Moore showed that $\textsf{BPFA}$ implies $2^{\al_0}=\al_2$ (\cite{Moore}).} The problem whether (consequences of) forcing axioms for classes of posets with small chain condition decide the size of the continuum does not seem to have received much attention in the literature so far.\footnote{Usually the focus has been on deriving $2^{\al_0}=\al_2$ from bounded forms of forcing axioms, that is forms of forcing axioms in which one considers only small maximal antichains but where the posets are allowed to have large antichains as well.} One place where the problem has been addressed is of course our \cite{ASP}. Prior to \cite{ASP}, M.\ Foreman and P.\ Larson showed in an unpublished note (\cite{FOREMAN-LARSON}) that the forcing axiom (relative to collections of $\al_1$--many dense sets) for the class of posets of size $\al_2$ preserving stationary subsets of $\o_1$ implies $2^{\al_0}=\al_2$. Several natural problems in this area remain open. For example it is not known whether the forcing axiom (relative to collections of $\al_1$--many dense sets) for the class of semi-proper posets of size $\al_2$ implies $2^{\al_0}=\al_2$, and the same is open for the forcing axiom for the class of all posets of size $\al_1$ preserving stationary subsets of $\o_1$, as well as for the forcing axiom for the class of all proper posets of size $\al_1$. Let us denote these two forcing axioms by, respectively, $\textsf{MM}(\o_1)$ and $\textsf{PFA}(\o_1)$. It is open whether or not $\textsf{MM}(\o_1)$ is equivalent to $\textsf{PFA}(\o_1)$,\footnote{On the other hand, $\textsf{PFA}(\o_1)$ is trivially equivalent to the forcing axiom for the class of semi-proper posets of size $\al_1$.} and even whether $\textsf{MM}(\o_1)$ has consistency strength above $\ZFC$.\footnote{$\textsf{PFA}(\o_1)$ can be easily forced over $\textsc{ZFC}$ model. Also, the existence of a non--proper poset of size $\al_1$ preserving stationary subsets of $\o_1$ is consistent. In fact, Hiroshi Sakai has constructed such a poset assuming a suitably strong version of $\diamondsuit_{\o_1}$ which holds in $\textbf L$ and which can always be forced.}

As we mentioned before, Theorem \ref{mainthm} shows in particular that the forcing axiom $\MA^{1.5}_{<\al_2}$ has the same consistency strength as $\ZFC$. Other articles dealing with the consistency strength of other (related) fragments of $\textsf{PFA}$ are \cite{Miyamoto}, \cite{HamkinsJohnstone}, \cite{NeemanSchimmerlingI}, and \cite{Neeman}.

The rest of the paper is organised as follows. In the following subsection we introduce the central notion of partial symmetric system of submodels and prove the main amalgamation results for these systems. Subsection \ref{the_forcing_construction} contains the construction of our iteration $\la\mtcl P_\a\,:\,\a\leq\k\ra$ ($\mtcl P_\k$ will witness Theorem \ref{mainthm}).
Finally, Subsection \ref{the_main_facts} contains proofs of the main facts about $\la\mtcl P_\a\,:\,\a\leq\k\ra$.
Theorem \ref{mainthm} follows then easily from these facts.

\subsection{Partial symmetric systems of submodels}\label{partial-symm-systems}

Throughout the paper, if $N$ and $N'$ are models for which that there is a (unique) isomorphism from $(N, \in)$ into $(N', \in)$, then we denote this isomorphism by $\Psi_{N, N'}$.  For this subsection let us fix a regular cardinal $\t$ and a predicate $T\sub H(\t)$. In this situation we will tend to refer to a structure $(N, \in, T\cap N)$ by the simpler expression $(N, \in, T)$.

\begin{definition}\label{hom}
Let $\{N_i\,:\,i<m\}$ be a finite set of countable subsets of $H(\t)$. We will say that \emph{$\{N_i\,:\,i<m\}$ is a partial $T$--symmetric system (of elementary substructures of $(H(\t), \in, T)$)} if and only if the following holds.

\begin{itemize}

\it[$(A)$] For every $i<m$, $(N_i, \in, T)$  is an elementary substructure of $(H(\t), \in, T)$.

\it[$(B)$] Given distinct $i$, $i'$ in $m$, if $\d_{N_i}=\d_{N_{i'}}$, then there is a unique isomorphism $$\Psi_{N_i, N_{i'}}:(N_i, \in, T)\into (N_{i'}, \in, T)$$

\noindent Furthermore, we ask that $\Psi_{N_i, N_{i'}}$ be the identity on $N_i\cap N_{i'}$.

\it[$(C)$] For all $i$, $i'$, $j$ in $m$, if $N_j\in N_i$ and $\d_{N_i}=\d_{N_{i'}}$, then there is some $j'<m$ such that $\Psi_{N_i, N_{i'}}(N_j)=N_{j'}$.

\it[$(D)$] For all $i$, $j<m$, if $\d_{N_i}<\d_{N_j}$, then there is some $i'<m$ such that
\begin{enumerate}
\it $\d_{N_{i'}}=\d_{N_i}$,
\it $N_{i'}\in N_j$, and
\it $N_i\cap N_j = N_{i'}\cap N_i$.
\end{enumerate}

\end{itemize}

\end{definition}

We will often talk about partial symmetric systems of elementary substructures, without mentioning $T$, or even just partial symmetric systems, in contexts where $T$ is understood or not relevant.

The present definition of symmetry is weaker than the one in \cite{ASP} in the following sense: In the definition of \emph{symmetric system} from \cite{ASP}, instead of condition (D) we had the following:

\begin{itemize}
\it[$(D)^+$] For all $i$, $j<m$, if $\d_{N_i}<\d_{N_j}$, then there is some $j'<m$ such that
\begin{enumerate}
\it $\d_{N_{j'}}=\d_{N_j}$,  and
\it $N_i\in N_{j'}$.
\end{enumerate}

\end{itemize}

It is easy to verify that every finite system of structures satisfying $(A)$, $(B)$, $(C)$ and $(D)^+$ (i.e., satisfying the original definition of symmetric system) satisfies $(A)$--$(D)$: Given $i$, $j$ as in the hypothesis of $(D)$, if $N_{j'}$ is such that $\d_{N_{j'}} =\d_{N_j}$ and $N_i\in N_{j'}$, then $N_{i'}:=\Psi_{N_{j'}, N_j}(N_i)$ witnesses $(D)$ relative to $N_i$ and $N_j$.

The following lemma is immediate.

\begin{lemma}\label{iso1}
Let $N$ and $N'$ be countable elementary substructures of $(H(\t), \in, T)$. Suppose $\mtcl N \in N$ is a partial $T$--symmetric system of elementary substructures of $H(\t)$ and $\Psi: (N, \in, T)  \into (N', \in, T)$ is an isomorphism. Then $\Psi(\mtcl N)= \Psi``\mtcl N$  is also a partial $T$--symmetric system of elementary substructures of $H(\t)$.
\end{lemma}

We will need three amalgamations lemmas for the proof of Theorem \ref{mainthm}. The first of the lemmas is Lemma \ref{iso2}. This lemma will be used in case $\a=0$ of the inductive proof of Lemma \ref{horribilis}.

\begin{lemma}\label{iso2}
Let $\mtcl N$ be a partial $T$--symmetric system of elementary substructures of $H(\t)$ and let $N\in\mtcl N$.
Then the following holds.
\begin{itemize}
\it[(1)] $\mtcl N\cap N$ is a partial $T$--symmetric system.

\it[(2)] Suppose $\mtcl N^\ast \in  N$ is a partial $T$--symmetric system such that $\mtcl N \cap N \sub \mtcl N^\ast$.

\noindent Let $$\mtcl M=\mtcl N\cup\bigcup\{\Psi_{N_0, N_1}``(\mtcl N^\ast \cap N_0) \mid N_0, N_1\in \mtcl N,\,\d_{N_0}=\d_{N_1}\leq \d_N\}$$

Then,
\begin{itemize}
\it[(i)] for every $N_0\in\mtcl N$ such that $\d_{N_0}\leq\d_N$,
\begin{itemize}
\it[(a)] $\mtcl M\cap N_0=\Psi_{N_1, N_0}``(\mtcl N^\ast\cap N_1)$ for every $N_1\in\mtcl N\cap(N\cup\{N\})$ such that $\d_{N_1}=\d_{N_0}$, and
\it[(b)] $\mtcl M\cap N_0$ is a partial $T$--symmetric system,
\end{itemize} and
\it[(ii)] $\mtcl M$ is a partial $T$--symmetric system.
\end{itemize}
\end{itemize}
\end{lemma}

\begin{proof} It is immediate to check that all conditions in the definition of partial symmetric system are inherited by $\mtcl N\cap N$, which gives (1).

Let us start the proof of (2). Let $N_0$ be as in the hypotheses of (i). Let $N_1\in\mtcl N\cap (N\cup\{N\})$ be of the same height as $\d_{N_0}$ and let us check that $\mtcl M\cap N_0=\Psi_{N_1, N_0}``(\mtcl N^\ast\cap N_1)$. Obviously, $\Psi_{N_1, N_0}``(\mtcl N^\ast\cap N_1)\sub N_0\cap\mtcl M$ by the construction of $\mtcl M$.  For the reverse inclusion, suppose $N_2$, $N_3\in\mtcl N$ are such that $\d_{N_2}=\d_{N_3}\leq\d_N$ and $M \in\mtcl N^\ast\cap N_2$ is such that $\Psi_{N_2, N_3}(M)\in N_0$. Since $M \in \mtcl N^\ast\sub N$, we may assume without loss of generality that $N_2=N$ if $\d_{N_2}=\d_N$ by conditions (B)--(D) in Definition \ref{hom} for $\mtcl N$, and $N_2\in N$ if $\d_{N_2}<\d_N$ by conditions (B) and (C).  We want to see that $\Psi_{N_2, N_3}(M)=\Psi_{N_1, N_0}(M')$ for some $M'\in\mtcl N^\ast\cap N_1$. We have two cases for this.

Suppose first that $\d_{N_1}\leq\d_{N_2}$. Let $N_0'=N_3$ if $\d_{N_0}=\d_{N_3}$ and,  if $\d_{N_0}<\d_{N_3}$, by condition (D) for $\mtcl N$ let $N_0'\in N_3\cap\mtcl N$ be such that $\d_{N_0'}=\d_{N_0}$ and $N_3\cap N_0=N_0\cap N_0'$. Then, if $\d_{N_0}<\d_{N_3}$, $\Psi_{N_3, N_2}(N_0')\in\mtcl N$ and $M\in  \mtcl N^\ast\cap\Psi_{N_3, N_2}(N_0')$ as $\Psi_{N_2, N_3}(M)\in N_0$.  In that case, $M':=\Psi_{\Psi_{N_3, N_2}(N_0'), N_1}(M)$ is as desired since $M'\in\mtcl N^\ast$, which is true because $N_1$ and $\Psi_{N_3, N_2}(N_0')$ are both in $\mtcl N\cap N\sub\mtcl N^\ast$, $M\in\mtcl N^\ast$, and $\mtcl N^\ast$ satisfies (C). In the case $\d_{N_0}=\d_{N_3}$ we can of course take $M'=\Psi_{N_2, N_1}(M)$.

Finally, suppose $\d_{N_2}<\d_{N_1}$. By (D) applied to $\mtcl N$, let $N_3'\in N_0\cap\mtcl N$ be such that $\d_{N_3'}=\d_{N_3}$ and $N_0\cap N_3=N_3\cap N_3'$. Then $\Psi_{N_2, N_3}(M)\in N_3\cap N_3'$, and therefore $\Psi_{N_2, N_3}(M)=\Psi_{N_2, N_3'}(M)=\Psi_{\Psi_{N_0, N_1}(N_3'), N_3'}(M')$ for $M' = \Psi_{N_2, \Psi_{N_0, N_1}(N_3')}(M)$. But $M'\in\mtcl N^\ast$ since $N_2$ and $\Psi_{N_0, N_1}(N_3')$ are both in $N\cap\mtcl N\sub \mtcl N^\ast$, $M\in \mtcl N^\ast$, and $\mtcl N^\ast$ satisfies (C). Finally, $\Psi_{\Psi_{N_0, N_1}(N_3'), N_3'}=\Psi_{N_1, N_0}\restr \Psi_{N_0, N_1}(N_3')$, which gives what we wanted.

Now, since $\mtcl M\cap N_0=\Psi_{N_1, N_0}``(\mtcl N^\ast\cap N_1)$ for some (equivalently, every) $N_1\in \mtcl N\cap (N\cup\{N\})$ of the same height as $N_0$, and since $\mtcl N^\ast\cap N_1$ is a partial symmetric system for every such $N_1$ by (1), $\mtcl M\cap N_0$ is a partial symmetric system by Lemma \ref{iso1}.

Let us proceed to the proof of (ii). It is clear that $\mtcl M$ satisfies condition (A) in Definition \ref{hom} thanks to Lemma \ref{iso1}.

To see that it satisfies condition (B), it suffices to show that if $Q_0$, $Q_1\in\mtcl M$ are such that $\d_{Q_0}=\d_{Q_1}\leq\d_N$, then $Q_0$ and $Q_1$ are isomorphic and the unique isomorphism between these structures fixes $Q_0\cap Q_1$. By symmetry of the situation we may assume that there are $N_0$, $N_1$, $M_0$ and $M_1$ in $\mtcl N$ such that $\d_{N_0}=\d_{N_1} \leq \d_{M_0}=\d_{M_1}\leq\d_N$ and such that $Q_0=\Psi_{N_0, N_1}(Q)$ and $Q_1=\Psi_{M_0, M_1}(Q^\ast)$ for some $Q$ and $Q^\ast$ in $\mtcl N^\ast$. By (B) for $\mtcl N^\ast$ there is a unique isomorphism $\Psi_{Q, Q^\ast}$ between $Q$ and $Q^\ast$ and $\Psi_{Q, Q^\ast}$ fixes $Q \cap Q^\ast$. Then $\Psi:=\Psi_{M_0, M_1}\circ \Psi_{Q, Q^\ast}\circ (\Psi_{N_1, N_0}\restr Q_0)$ is of course an isomorphism between $Q_0$ and $Q_1$. To see that $\Psi$ fixes $Q_0\cap Q_1$, suppose $x\in Q_0\cap Q_1$. Let $N_1'=M_1$ if $N_1$ and $M_1$ have the same height and, if $\d_{N_1}<\d_{M_1}$, let $N_1'\in M_1\cap\mtcl N$ be such that $N_1'\cap N_1=N_1\cap M_1$. Then $Q':=\Psi_{N_0, N_1'}(Q)\in\mtcl M\cap N_1'\sub\mtcl M\cap M_1$.  Since $x\in N_1'\cap N_1$, we have $\Psi_{N_1', N_1}(x)=x$. In particular $x\in Q'$ since $\Psi_{N_1', N_1}(Q')=Q_0$ and $x\in Q_0$. But then $x\in Q'\cap Q_1$, and since both $Q'$ and $Q_1$ are in $\mtcl M\cap M_1$ and $\mtcl M\cap M_1$ is a partial symmetric system by (i), the unique isomorphism $\Psi_{Q', Q_1}$ between $Q'$ and $Q_1$ fixes $x$. But then $\Psi=\Psi_{Q', Q_1}\circ(\Psi_{N_1, N_1'}\restr Q_0)$ since $\Psi_{N_1', N_1}(Q') = Q_0$, and therefore $\Psi(x)=x$ since $\Psi_{N_1', N_1}(x)=x$.

Let us verify now (C) for $\mtcl M$. Let $N_0$, $N_1\in\mtcl N$ be of the same height and let $M\in\mtcl M\cap N_0$. We want to see that $\Psi_{N_0, N_1}(M)\in\mtcl M$.  Without loss of generality we may assume $\d_M<\d_N$. The case $\d_{N_0}\leq \d_N$ follows immediately from part (i) (a). For the case $\d_{N_0}>\d_N$, let $N'\in N_0\cap \mtcl N$ be such that $\d_{N'} \leq \d_N$ and $M \in N'$ (since $\mtcl N$ satisfies (C), we may also assume that $M \in \mtcl M \setminus  \mtcl N$). Then $\Psi_{N_0, N_1}(N')\in \mtcl N\cap N_1$ by (C) for $\mtcl N$, and we are done again by the previous case since $\Psi_{N', \Psi_{N_0, N_1}(N')}(M)=\Psi_{N_0, N_1}(M)$.

We still need to consider the situation when $N_0$ and $N_1$ are in $\mtcl M\setminus\mtcl N$, $\d_{N_0}=\d_{N_1}$, $M\in\mtcl M\cap N_0$, and at least one of $N_0$, $N_1$ is not in $\mtcl N$.  Suppose first $N_0\notin \mtcl N$ and $N_1\notin \mtcl N$. Then there is some $M_0\in\mtcl N$ such that $\d_{M_0}\leq\d_N$ and $N_0\in M_0$ and such that $N_0=\Psi_{\ov{M}_0, M_0}(\ov{N}_0)$ for some $\ov{M}_0\in \mtcl N\cap (N\cup\{N\})$ of the same height as $M_0$ and some $\ov{N}_0\in\mtcl N^\ast\cap \ov{M}_0\in \ov{N}_0$. We can also find $M_1\in\mtcl N$ such that $\d_{M_1}\leq\d_N$ and $N_1\in M_1$ and such that $N_1=\Psi_{\ov{M}_1, M_1}(\ov{N}_1)$ for some $\ov{M}_1\in \mtcl N\cap (N\cup\{N\})$ of the same height as $M_1$ and some $\ov{N}_1\in\mtcl N^\ast\cap \ov{M}_1$. Let $\e\in\{0, 1\}$ with $\d_{M_\e}$ minimal, let $M_{\e}'=M_{1-\e}$ if $\d_{M_0}=\d_{M_1}$ and, if $\d_{M_\e}<\d_{M_{1-\e}}$, let $M_{\e}'\in\mtcl N\cap M_{1-\e}$ be of the same height as $M_\e$ and such that $M_{1-\e}\cap M_\e= M_\e\cap M_{\e}'$. Then $\Psi_{M_\e, M_{\e}'}(N_\e)\in M_{1-\e}\cap\mtcl M$, and we can finish thanks to part (i) (b) by looking at  $M_{1-\e}\cap\mtcl M$ and taking the appropriate isomorphic copies. Finally, in the cases when $N_0\notin \mtcl N$ and $N_1\in\mtcl N$ and when $N_0\in\mtcl N$ and $N_1\notin\mtcl N$, we can run an easier variant of the argument above.

Finally let us verify (D) for $\mtcl M$. Suppose $M_0$, $M_1$ are both in $\mtcl M$ and $\d_{M_0}<\d_{M_1}$. We want to see that there is some $M'_0\in\mtcl M\cap M_1$ of the same height as $M_0$ and such that $M'_0\cap M_0=M_0\cap M_1$. We may of course assume that $M_0$ and $M_1$ are not both in $\mtcl N$.

Suppose $M_0\notin \mtcl N$ and $M_1\in\mtcl N$. Let $N_0\in\mtcl N$ be such that $\d_{N_0}\leq\d_N$ and $M_0\in N_0$. If $\d_{N_0}=\d_{M_1}$, then we may take $M'_0=\Psi_{N_0, M_1}(M_0)$, and if $\d_{N_0}<\d_{M_1}$ we may take $M'_0= \Psi_{N_0, N'_0}(M_0)$, where $N'_0\in\mtcl N\cap M_1$ is of the same height as $N_0$ and such that $N'_0\cap N_0=N_0\cap M_1$. If $\d_{N_0}>\d_{M_1}$, then there is some $M'_1\in\mtcl N\cap N_0$ of the same height as $M_1$ and such that $M_1\cap N_0=M'_1\cap M_1$.
By part (i) (b) for $\mtcl M\cap N_0$ there is some $M_0^\ast\in \mtcl M\cap M_1'$ of the same height as $M_0$ and such that $M_0^\ast\cap M_0=M_1'\cap M_0$. Now it is easy to check that $\Psi_{M'_1, M_1}(M_0^\ast)\in\mtcl M$ satisfies the desired conclusion.

Next, suppose $M_0\in\mtcl N$ and $M_1\notin\mtcl N$. Let $N_1\in\mtcl N$ be such that $\d_{N_1}\leq\d_N$ and $M_1\in N_1$.  Since $\d_{M_0}<\d_{N_1}$, there is some $M_0^\ast\in \mtcl N\cap N_1$ of the same height as $M_0$ and such that $M_0^\ast\cap M_0=M_0\cap N_1$. By part (i) (b) for $\mtcl M\cap N_1$, there is some $M'_0\in \mtcl M\cap M_1$ of the same height as $M_0^\ast$ and such that $M_0^\ast\cap M'_0=M_0^\ast\cap M_1$. Now it is easy to check that $M'_0$ is as desired.

Finally, suppose that none of $M_0$ and $M_1$ is in $\mtcl N$. Let $N_0$, $N_1\in\mtcl N$ be such that $\d_{N_0}\leq\d_N$, $\d_{N_1}\leq\d_N$, $M_0\in N_0$ and $M_1\in N_1$. Suppose $\d_{N_0}\leq\d_{N_1}$. Let $N'_0=N_1$ if $N_0$ and $N_1$ have the same height and, if $\d_{N_0}<\d_{N_1}$, let $N'_0\in \mtcl N\cap N_1$ be of the same height as $N_0$ and such that $N_0\cap N_1=N_0\cap N'_0$. Then, $M^\ast_0:=\Psi_{N_0, N'_0}(M_0)\in\mtcl M\cap N_1$. By (i) (b) for $\mtcl M\cap N_1$ we may now find $M_0'\in M_1$ of the same height as $M_0'$ and such that $M_1\cap M_0^\ast=M_0^\ast\cap M_0'$.  It is then easy to check that $M_0'$ is as desired. The other case is when $\d_{N_1}<\d_{N_0}$. Let $N'_1\in\mtcl N\cap N_0$ be such that  $\d_{N'_1}=\d_{N_1}$ and $N_0\cap N_1=N_1\cap N'_1$. By part (i) (b) for $\mtcl M\cap N_0$ we may pick $M_0^\ast\in N'_1\cap\mtcl M$ of the same height as $M_0$ and such that $N'_1\cap M_0=M^\ast_0\cap M_0$. Let $M^{\ast\ast}_0=\Psi_{N'_1, N_1}(M_0^\ast)\in\mtcl M$. Now, by (i) (b) for $\mtcl M\cap N_1$ we may find $M_0'\in\mtcl M\cap M_1$ of the same height as $M_0^{\ast\ast}$ and such that $M_0'\cap M_0^{\ast\ast}=M_0^{\ast\ast}\cap M_1$, and it is easy to check that $M_0'$ is as desired.
\end{proof}

Next comes our second amalgamation lemma for partial symmetric systems. It will be used in the proof of Lemma \ref{cc}.

\begin{lemma}\label{iso3}
Let $\mtcl N_0$ and $\mtcl N_1$ be partial $T$--symmetric systems of elementary substructures of $H(\t)$.
Suppose that $(\bigcup\mtcl N_0)\cap(\bigcup\mtcl N_1)=R$ and that, for some $m<\o$, there are enumerations $(N^0_i)_{i <m}$ and $(N^1_i)_{i <m}$ of $\mtcl N_0$ and $\mtcl N_1$, respectively, together with an isomorphism $\Psi$ between $\la \bigcup\mtcl N_0,\in, T, R, N^0_i\ra_{i<m}$ and $\la \bigcup\mtcl N_1,\in, T, R, N^1_i\ra_{i<m}$ which is the identity on $R$.  Then $\mtcl N_0\cup\mtcl N_1$ is a partial $T$--symmetric system.
\end{lemma}

\begin{proof}
$\mtcl N_0\cup\mtcl N_1$ obviously satisfies condition (A) in Definition \ref{hom}.

For condition (B), we must show that if $i_0$, $i_1< m$ are such that $\d_{N^0_{i_0}}=\d_{N^1_{i_1}}$, then the isomorphism
$\Psi_{N^0_{i_0}, N^1_{i_1}}:=\Psi\circ\Psi_{N^0_{i_0}, N^0_{i_1}}$ fixes $N^0_{i_0}\cap N^1_{i_1}$.
Now, if $x\in N^0_{i_0}\cap N^1_{i_1}$, then $x\in R\cap N^0_{i_0}$, which implies that $\Psi(x)=x\in N^1_{i_0}\cap N^1_{i_1}$ as $\Psi$ is an isomorphism between the structures $\la\bigcup\mtcl N_0,\in, N^0_i\ra_{i<m}$ and $\la \bigcup\mtcl N_1,\in, N^1_i\ra_{i<m}$. But then $x \in N^0_{i_0}\cap N^0_{i_1}$ again by the fact that $\Psi$ is an isomorphism between $\la\bigcup\mtcl N_0,\in, N^0_i\ra_{i<m}$ and $\la \bigcup\mtcl N_1,\in, N^1_i\ra_{i<m}$, which implies that $\Psi_{N^0_{i_0}, N^0_{i_1}}(x)=x$ and hence that $((\Psi\restr N^0_{i_1}) \circ\Psi_{N^0_{i_0}, N^0_{i_1}})(x) = \Psi_{N^0_{i_0}, N^1_{i_1}}(x)=x$.

A similar argument establishes condition (C) for $\mtcl N_0\cup\mtcl N_1$:
Let $M_1$, $M_2$, $M_3 \in \mtcl N_0\cup\mtcl N_1$ be such that $M_2 \in M_1$ and $\delta_{M_1} = \delta_{M_3}$. We must prove that $\Psi_{M_1, M_3}(M_2)$ is also in $\mtcl N_0\cup\mtcl N_1$. We may clearly assume that there are indices $i$, $j \in \{1,2,3\}$ such that $M_i \in \mtcl N_0$  and $M_j \in \mtcl N_1$. The case when $M_1$ and $M_3$ are both in $\mtcl N_1$ and $M_2$ is in $\mtcl N_0$ can be treated as follows. First note that there are $i_1$ and $i_2$ such that $M_1= N^1_{i_1}$ and $M_2 = N^{0}_{i_2}$. As $M_2 \in M_1$ and $\Psi$ is an isomorphism fixing $R$ (in particular, $M_2$), $N^{0}_{i_2} \in N^{0}_{i_1}$. But $\Psi_{N^0_{i_2}, N^1_{i_2}}=\Psi_{N^0_{i_1}, N^1_{i_1}} \restr N^0_{i_2}$ and this isomorphism also fixes $M_2$. So $M_1$, $M_2= N^{1}_{i_2}$ and $M_3$ are in $\mtcl N_1$. The last case that needs to be considered is when $M_3\in\mtcl N_0$ and $M_1\in\mtcl N_1$. Just as before, we can ensure the existence of $i_1$, $i_2$ and $i_3$ such that $M_3 = N^{0}_{i_3}$, $M_1=N^1_{i_1}$ and $M_2 =N^{1}_{i_2}$. Finally, let $i_4$ be such that $N^1_{i_4}=\Psi_{N^1_{i_1}, N^1_{i_3}}(N^1_{i_2})$ and note that $N^0_{i_4}= \Psi_{M_1,M_3}(M_2)$.

Finally, as to condition (D) for $\mtcl N_0\cap\mtcl N_1$, it suffices to show that if $i_0$, $i_1<m$ are such that $\d_{N^0_{i_0}}<\d_{N^1_{i_1}}$, then there is some $i$ such that $\d_{N^1_i}=\d_{N^0_{i_0}}$, $N^1_i \in N^1_{i_1}$, and $N^1_{i_1}\cap N^0_{i_0} = N^0_{i_0}\cap N^1_i$. For this, by condition (D) for $\mtcl N_0$ we may find $i<m$ such that $\d_{N^0_i}=\d_{N^0_{i_0}}$, $N^0_i\in N^0_{i_1}$, and $N^0_{i_1}\cap N^0_{i_0}=N^0_i\cap N^0_{i_0}$. Then, since $\Psi$ is an isomorphism, $N^1_i$, $N^0_i$ and $N^0_{i_0}$ have the same height and $N^1_i \in N^1_{i_1}$ . To see that  $N^1_{i_1}\cap N^0_{i_0}=N^1_i\cap N^0_{i_0}$, note that any set in either $N^1_{i_1}\cap N^0_{i_0}$ or $N^1_i\cap N^0_{i_0}$ is in $R$, and use the fact that $\Psi$ is the identity on $R$ together with the fact that $N^0_{i_1}\cap N^0_{i_0}=N^0_i\cap N^0_{i_0}$.
\end{proof}

Our final amalgamation lemma is the following. This lemma will be used in the proof of Lemma \ref{refl}.

\begin{lemma}\label{iso4}
Let $\mtcl N_0$, $\mtcl N_1$ and $\mtcl N_2$ be partial $T$--symmetric systems of elementary substructures of $H(\t)$ such that $\mtcl N_0\sub\mtcl N_2$.
Suppose that $$(\bigcup\mtcl N_0)\cap(\bigcup\mtcl N_1)=R$$ and that, for some $m<\o$, there are enumerations $(N^0_i)_{i <m}$ and $(N^1_i)_{i <m}$ of $\mtcl N_0$ and $\mtcl N_1$, respectively, together with an isomorphism $\Psi$ between $\la \bigcup\mtcl N_0,\in, T, R, N^0_i\ra_{i<m}$ and $\la \bigcup\mtcl N_1,\in, T, R, N^1_i\ra_{i<m}$ which is the identity on $R$.  Suppose, in addition, that $$(\bigcup\mtcl N_2)\cap(\bigcup\mtcl N_1)=R$$ Then $\mtcl N_1\cup\mtcl N_2\cup\mtcl N$, where $$\mtcl N =\{\Psi_{N_0, N_1}(M)\,\mid\, M\in \mtcl N_2,\,M\in N_0, N_0\in\mtcl N_0, \,N_1\in\mtcl N_1,\,\d_{N_0}=\d_{N_1}\},$$ is a partial $T$--symmetric system.
\end{lemma}

\begin{proof}
Note that $\mtcl M=\mtcl N_0\cup\mtcl N_1$ is a partial symmetric system by Lemma \ref{iso3}. Note also that $\mtcl N_1\cup\mtcl N_2\cup\mtcl N$ can be written as $$\mtcl M_m\cup\{M\in\mtcl N_2\,\mid\, \textrm{ there is no }N\in \mtcl N_0\textrm{ such that }M\in N\}$$ for some finite $\sub$--increasing sequence $(\mtcl M_i)_{i\leq m}$ of partial symmetric systems such that $\mtcl M_0=\mtcl M$ and such that for all $i<m$ there is some $N\in\mtcl N_0$ such that $$\mtcl M_{i+1}=\mtcl M_i\cup \bigcup\{\Psi_{N_0, N_1}``(\mtcl N_2\cap N)  \mid N_0, N_1\in \mtcl M_i,\,\d_{N_0}=\d_{N_1}\leq \d_N\}$$ By Lemma \ref{iso2}, each $\mtcl M_i$ (for $i\leq m$) is a partial symmetric system. Also note that in fact $(\mtcl M_i)_{i\leq m}$ can be set up in such a way that there are $N^0,\ldots, N^{m-1}\in\mtcl N_0$ such that for every $i<m$, $\mtcl M_{i+1}$ is $$\mtcl M_i\cup\bigcup_{k\leq i}\{\Psi_{N_0, N_1}``(\mtcl N_2\cap N^k) \mid N_0\in\mtcl N_0, N_1\in \mtcl N_0\cup \mtcl N_1, \d_{N_0}=\d_{N_1}\leq \d_{N^k}\}$$ Indeed, start with any $N^0\in\mtcl N_0$ such that $N^0\cap(\mtcl N_2\setminus\mtcl N_0)\neq\emptyset$ and build the corresponding $\mtcl M_1$. Note that $\mtcl M_1\cap N\sub \mtcl N_2\cap N$ for every $N\in\mtcl N_0$ since condition (C) in Definition \ref{hom} holds for $\mtcl N_2\supseteq\mtcl N_0$ and since $\Psi_{N_0, N_1}$ is the identity on $R$ whenever $N_0\in\mtcl N_0$ and $N_1\in\mtcl N_1$ are such that $\d_{N_0}=\d_{N_1}$. Now we pick, if possible, any $N^1\in\mtcl N_0$ such that $N^1\cap(\mtcl N_2\setminus\mtcl M_1)\neq \emptyset$, and build the corresponding $\mtcl M_2$. We iterate this process until we get nothing new.
 Let us verify now that $\mtcl N_1\cup\mtcl N_2\cup\mtcl N$ satisfies conditions (A)--(D) in Definition \ref{hom}.

Condition (A) is obviously true. For condition (B), since it holds for $\mtcl N_2\supseteq\mtcl N_0$ it suffices to show, by the above representation of $(\mtcl M_i)_{i\leq m}$, that if $M\in\mtcl N_2$, $\ov M\in\mtcl N_2$ is such that $\d_{\ov M}=\d_M$, and $N_0\in\mtcl N_0$, $N_1\in\mtcl N_1$ are such that $\ov M\in N_0$ and $\d_{N_0} = \d_{N_1}$, then $\Psi:=\Psi_{N_0, N_1}\circ \Psi_{M, \ov M}$ fixes $M\cap\Psi_{N_0, N_1}(\ov M)$ (all other cases are easily proved using our hypothesis on $\mtcl N_0$ and $\mtcl N_1$ and the fact that $\mtcl N_2$ is a partial symmetric system extending  $\mtcl N_0$). Let $x\in M\cap\Psi_{N_0, N_1}(M')$ and note that $x\in M\cap R$. By our assumption on $\mtcl N_0$ and $\mtcl N_1$ there is some $N_1'\in\mtcl N_0$ such that $N_1'\cap R=N_1\cap R$. By condition (D) in Definition \ref{hom} for $\mtcl N_2$ there is some $M'\in\mtcl N_2\cap N_1'$ such that $\d_{M'}=\d_M$ and $M\cap N_1'=M\cap M'$. But then $\Psi(x)=(\Psi_{N_1', N_1}\circ\Psi_{M, M'})(x)=x$ since $x\in M\cap M'\cap N_1\cap N_1'$.

As to condition (C) it suffices to argue, again thanks to the above representation of $(\mtcl M_i)_{i\leq m}$,  that if $M\in\mtcl N_2$, $\ov M\in\mtcl N_2$ is such that $\d_{\ov M}=\d_M$, $N_0\in\mtcl N_0$, $N_1\in\mtcl N_1$ are such that $\ov M\in N_0$ and $\d_{N_0} = \d_{N_1}$, and $Q\in M\cap \mtcl N_2$, then $\Psi(Q)\in \mtcl N_1\cup\mtcl N_2\cup\mtcl N$ for $\Psi= \Psi_{N_0, N_1}\circ \Psi_{M, \ov M}$ (again, all remaining cases can be proved easily using our hypotheses). But $\Psi_{M, \ov M}(Q)\in\mtcl N_2$ by (C) for $\mtcl N_2$, and therefore $\Psi_{N_0, N_1}(\Psi_{M, \ov M}(Q))\in\mtcl N$.

Finally we verify condition (D). Suppose $M$ and $Q$ are both in $\mtcl N_2$ and $\d_M<\d_Q$. Suppose first that there are $N_0\in\mtcl N_0$, $N_1\in\mtcl N_1$ of the same height such that $M\in N_0$ and let us show that there is some $\tld M\in Q\cap\mtcl N_2$ of the same height as $M$ such that $\Psi_{N_0, N_1}(M)\cap Q=\tld M\cap \Psi_{N_0, N_1}(M)$. For this, note that by our hypothesis on $\mtcl N_0$ and $\mtcl N_1$ there is some $N_1'\in\mtcl N_0$ such that $N_1'\cap R = N_1\cap R$. By (D) for $\mtcl N_2\supseteq \mtcl N_0$ there is some $\tld M\in\mtcl N_2\cap Q$ of the same height as $M$ and such that $\Psi_{N_0, N_1'}(M)\cap Q=\Psi_{N_0, N_1'}(M)\cap \tld M$. But now it is easy to check that $\Psi_{N_0, N_1}(M)\cap Q =\Psi_{N_0, N_1}(M)\cap\tld M$ using the fact that if $x\in \Psi_{N_0, N_1}(M)\cap Q$, then $x\in R$ (which is true since $(\bigcup \mtcl N_2)\cap(\bigcup\mtcl N_1)=R$).  Finally suppose there are $N_0\in\mtcl N_0$, $N_1\in\mtcl N_1$ of the same height such that $Q\in N_0$ and let us show that there is some $M'\in \Psi_{N_0, N_1}(Q)\cap\mtcl N$ such that $M\cap\Psi_{N_0, N_1}(Q)=M\cap M'$: By (D) for $\mtcl N_2$ there is some $\tld M\in\mtcl N_2\cap Q$ of the same height as $M$ and such that $M\cap Q = M\cap \tld M$. Then it is easy to see that $M'=\Psi_{N_0, N_1}(\tld M)\in\mtcl N\cap \Psi_{N_0, N_1}(Q)$ is as desired, again using the fact that $(\bigcup \mtcl N_2)\cap(\bigcup\mtcl N_1)=R$. The remaining cases in the verification of (D) are easier as usual.
\end{proof}

\subsection{The forcing construction}\label{the_forcing_construction}

We start our forcing preparations fixing $\vec X=\la X_\a\,:\,\a\in \k,\,\cf(\a)\geq\o_1\ra$, a $\diamondsuit(\{\a\in \k,\,\cf(\a)\geq\o_1\})$--sequence. Note that the existence of such a diamond sequence implies $2^{<\k} = \k$. Let $\Phi$ be a bijection between $\k$ and $H(\k)$. Let also $\lhd$ be a well--order of $H(\k^+)$ in order type $2^\k$. $\Phi$ exists since $\av H(\k)\av=\k$ by  $2^{<\k}=\k$, and $\lhd$ exists since $\av H(\k^+)\av = 2^\k$.

We will define an iteration $\la\mtcl P_\a\,:\,\a\leq\k\ra$, together with a function $\Upsilon :\k\into H(\k)$ such that each $\Upsilon(\a)$ is a $\mtcl P_\a$--name in $H(\k)$.

Let $\la \t_\a\,:\,\a\leq \k\ra$ be the strictly increasing sequence of regular cardinals defined as $\t_0 =\av 2^{\k}\av^{+}$ and $\t_\a=|2^{\textsf{sup}\{\t_\beta\,:\,\b < \a \}}|^{+}$ if $\a>0$. For each $\a\leq\k$ let $\mathfrak N^\ast_\a$ be the collection of all countable elementary substructures of $H(\t_\a)$ containing $\Phi$, $\lhd$ and $\la \t_\b\,:\,\b<\a\ra$.
Let also $\mathfrak N_\a=\{N^\ast\cap H(\k)\,:\,N^\ast\in\mathfrak N^\ast_\a\}$ and note that if $\alpha < \beta$, then $\mathfrak N^\ast_\a$ belongs to all members of $\mathfrak N^\ast_\b$ containing the ordinal $\alpha$.

The properness of each $\mtcl P_\a$ will be witnessed by the club $\mathfrak N^\ast_\a$ (see Lemma \ref{horribilis}).  The proof of this will be by induction on $\a$.

Let us proceed to the definition of $\la\mtcl P_\a\,:\,\a\leq \k\ra$ now. Conditions in $\mtcl P_0$ are pairs of the form $(\emptyset, \D)$, where

\begin{itemize}

\it[$(A)$] $\D$ is a finite set of  ordered pairs of the form $(N, 0)$ such that $\dom(\D)$ is a partial $\Phi$--symmetric system.

\end{itemize}

Given $\mtcl P_0$--conditions $q_\e=(\emptyset, \D_\e)$ for $\e\in\{0, 1\}$, $q_1$ extends $q_0$ if and only if

\begin{itemize}

\it[$(B)$] $\dom(\D_0)\sub \dom(\D_1)$

\end{itemize}

\begin{notation} If $q$ is an ordered pair $(F, \D)$ such that $F$ is a function and $\D$ is a relation, we denote $F$ by $F_q$ and $\D$ by $\D_q$. In this context, we will use $\supp(q)$ to denote the domain of $F_q$ ($\supp(q)$ stands for the \emph{support} of $q$). Also, if $q$ is as above and $\x$ is an ordinal, \emph{the restriction of $q$ to $\x$}, denoted by $q\av_\x$, is defined
as the pair $$q\av_\x:=(F_q\restr\x,\,\{(N, \min\{\b,\,\x\})\,:\,(N, \b)\in \D_q\})$$
\end{notation}

Let $\a\leq \k$, $\a>0$, and suppose that we have defined $\mtcl P_\x$ for all $\x<\a$. Suppose also
that if $\x<\a$, then $\mtcl P_\x\sub H(\k)$, and if $q \in \mtcl P_\x$, then $q$ is an ordered pair of the form $(F_q, \Delta_q)$, where

\begin{itemize}

\it $F_q$ is a finite function with domain included in $\x$, and

\it $\Delta_q$  is a finite relation $\{(N_i, \g_i) \mid i \in n \}$ such that $\dom(\D_q)=\{N_i\mid i<n\}$ is a partial $\Phi$--symmetric system and, for all $i$, $\g_i$ is an ordinal such that $\g_i\leq\x$.
\end{itemize}

It will be convenient to consider the following technical variant of the $\al_{1.5}$--c.c.\ in the context of our construction. In this definition, and throughout the paper, if $\mtbb P\sub H(\k)$ is a poset, $N\sub H(\k)$, and $G\sub\mtbb P$ is a filter, then by $N[G]$ we mean $\{\tau_G\,:\,\tau\in N,\,\tau\textrm{ a }\mtbb P\textrm{--name}\}$. $\mtbb P$ might certainly fail to be a member of $N$ and might not even be definable there.

\begin{definition} (In $\V[G]$ for a $\mtcl P_\x$--generic filter $G$, $\x < \a$) A poset $\mtcl Q\sub H(\k)^{\V}$ is \emph{$\aleph_{1.5}$--c.c.\ relative to $\V$ and $G$} if and only if there is a club $D$ of $[H(\k)^{\V}]^{\al_0}$, $D$ belonging to $\V$, with the following property:

If \begin{enumerate} \it $p\in \mtcl Q$, \it $\{N_i\,:\, i \in n\} \subseteq D$ is finite, \it $p\in N_j[G]$ for some $j$ such that $\d_{N_j}=\min\{\d_{N_i}\,:\,i<n\}$, and \it $\{(N_i, \x)\,:\,i<m\}\sub\D_u$ for some $u\in G_\x$, \end{enumerate} then there is an extension of $p$ which is $(N_i[G],\, \mtcl Q)$--generic for all $i$.

\end{definition}

If $\a=\a_0+1$, then let $\Upsilon(\a_0)$ be a $\mtcl P_{\a_0}$--name in $H(\k)$ for an $\aleph_{1.5}$--c.c.\ poset relative to $\V$ and $\dot G_{\a_0}$
such that $\Upsilon(\a_0)$ is (say) the canonical $\mtcl P_{\a_0}$--name for trivial forcing on $\{0\}$ unless $X_{\a_0}$ is defined and codes (via some fixed reasonable translating function $\p$)\footnote{$\p$ can be taken to be for example the following surjection $\p:\mtcl P(\k)\into H(\k^+)$: if $a\in H(\k^+)$, then $\p(X) = a$ if and only if $X\sub\k$ codes a structure $(\k', E)$ isomorphic to $(\TC(\{a\}), \in)$ (for some unique cardinal $\k'\leq\k$). We will denote this isomorphism by $\p_X$.}  a $\mtcl P_{\a_0}$--name $\dot X$. In that case, $\Upsilon(\a_0)$ is a $\mtcl P_{\a_0}$--name in $H(\k)$ for an $\aleph_{1.5}$--c.c.\ poset relative to $\V$ and $\dot G_{\a_0}$
such that $\Upsilon(\a_0)$ is $\dot X$ if $\dot X$ is such a poset and such that $\Upsilon(\a_0)$ is trivial forcing on $\{0\}$ if $\dot X$ is not such a poset.

Let $0< \alpha \leq \kappa$. The definition of $\mtcl P_\a$ is as follows (regardless of whether $\alpha$ is a successor or a limit ordinal). Conditions in $\mtcl P_\a$ are pairs of the form $q= (F_q, \D_q)$ with the following properties.

\begin{itemize}

\it[$(C0)$] $F_q$ is a finite function whose domain is included in $\a$.

\it[$(C1)$] $\D_q$ is a finite set of pairs $(N, \g)$ with $\g\leq \min\{\a, \textsf{sup}(N\cap \kappa)\}$.

\it[$(C2)$] For every $\x<\a$, the restriction $q\av_\x$ of $q$ to $\x$ is a condition in $\mtcl P_\x$.

\it[$(C3)$] If $\x \in \dom(F_q)$, then $F_q(\x) \in H(\k)$ and $q\av_\x\Vdash_{\mtcl P_\x} \check{F}_q(\x)\in\Upsilon(\x)$.\footnote{As a convenient simplification, we will tend to omit inverted circumflexes (i.e., we will tend to write $x$ rather than $\check x$)
in cases where we clearly deal with canonical names.}

\it[$(C4)$] If $\x\in \dom(F_q)$, $N \in \mathfrak N_{\x+ 1}$, and $(N,\b) \in  \D_q$ for some $\beta \geq \x+1$, then $q\av_\x$ forces in $\mtcl P_\x$ that $F_q(\x)$ is $(N[\dot G_\x], \Upsilon(\x))$--generic.
\end{itemize}

Given conditions $q_\e=(F_\e, \D_e)$ (for $\e\in\{0, 1\}$) in $\mtcl P_\a$, $q_1$ extends $q_0$ if and only if the following holds:

\begin{itemize}

\it[$(D1)$] $q_1\av_\x \leq_\a q_0\av_\x$ for all $\x<\a$.

\it[$(D2)$] $\dom(F_0)\subseteq \dom(F_1)$ and if $\x \in \dom(F_0)$, then $q_1\av_{\x}$ forces in $\mtcl P_{\x}$ that $F_1(\x)$ $\Upsilon(\x)$--extends $F_0(\x)$.

\it[$(D3)$] If $(N, \b) \in \D_0$, then there exists $\widetilde{\b}\geq \b$ such that $(N, \widetilde{\b}) \in \D_1$.

\end{itemize}

We define also the notion of being an $\aleph_{1.5}$--c.c.\ poset relative to $\V$ and $G$ for any $\mtcl P_\k$--generic filter $G$ in the natural way: A poset $\mtcl Q\sub H(\k)^{\V}$ is \emph{$\aleph_{1.5}$--c.c.\ relative to $\V$ and $G$} if and only if there is a club $D$ of $[H(\k)^{\V}]^{\al_0}$, $D$ belonging to $\V$, with the following property:

If \begin{enumerate} \it $p\in \mtcl Q$, \it $\{N_i\,:\, i \in n\} \subseteq D$ is finite, \it $p\in N_j[G]$ for some $j$ such that $\d_{N_j}=\min\{\d_{N_i}\,:\,i<n\}$, and \it $\{(N_i, \textsf{sup}(N_i\cap \k))\,:\,i<m\}\sub\D_u$ for some $u\in G_\x$, \end{enumerate} then there is an extension of $p$ which is $(N_i[G],\, \mtcl Q)$--generic for all $i$.

\subsection{The Main Facts}\label{the_main_facts}

We are going to prove the relevant properties of the forcings $\mtcl P_\a$. Theorem \ref{mainthm} will follow immediately from them.

Our first lemma is immediate from the definitions.

\begin{lemma}\label{suborder} $\mtcl P_\k=\bigcup_{\b<\k}\mtcl P_\b$, and $\emptyset\neq\mtcl P_\a\sub \mtcl P_\b$ for all $\a\leq\b\leq\k$.\end{lemma}

An inductive verification shows that if $q$ is a condition in $\mtcl P_\alpha$ ($\alpha \leq \kappa$), $(N, \varrho)\in\Delta_q$, and $\rho$ is an ordinal below $\varrho$, then $(F_q, \Delta_q  \cup \{(N, \rho)\})$ is also a condition in $\mtcl P_\alpha$. In fact, the original condition $q$ is clearly equivalent to this condition. Also note that if $r$ is a condition in $\mtcl P_\beta$ and $\alpha \leq \beta \leq \kappa$, then the relations $\Delta_r$ and $\Delta_{r\av_\a}$ have the same domain. In particular, if $q\in\mtcl P_\alpha$, $\{q^i \,:\, i \in I \}$ is a finite subset of $\mtcl P_\beta$, and $q \leq_\alpha q^i \av_\a$ for all $i \in I$, then $\dom(\Delta_{q^i})= \dom (\Delta_{q^i \av_\a}) \subseteq \dom(\Delta_{q}) = \dom(\Delta_{q \av_0})$ (the inclusion follows from clauses $(B)$ and $(D3)$ in the definition of our forcing). To sum up,  $\Delta_q \cup \bigcup\{\Delta_{q^i\av_\a}\,:\, i \in I \}$ and $\Delta_q$ have he same domain, and this domain is of course a partial $\Phi$--symmetric system. Again by $(D3)$, if $q \leq_\alpha q^i \av_\a$ and $(N, \rho)$ is in $\Delta_{q^i\av_\a}$, then there is $\varrho \geq \rho$ such that $(N, \varrho)$ is in $\Delta_q$. So, after repeating a finite number of times the first observation of this paragraph we have the following result:

\begin{lemma}\label{preorder}
Let $\alpha$ and $\beta$ be ordinals such that $\alpha \leq \beta \leq \kappa$. If $q$ is in $P_\alpha$, $\{q^i \,:\, i \in I \}$ is a finite subset of $\mtcl P_\beta$ and $q \leq_\alpha q^i \av_\a$ for all $i \in I$, then $q$ and $(F_q, \Delta_q \cup \bigcup\{\Delta_{q^i\av_\a}\,:\, i \in I \})$ are two equivalent conditions in the forcing $\mtcl P_\alpha$.
\end{lemma}

Lemma \ref{compll} shows that $\la\mtcl P_\a\,:\,\a\leq\k\ra$ is a forcing iteration in a broad sense.

\begin{lemma}\label{compll}
If $q=(F_q, \D_{q})\in \mtcl P_\a$, $r=(F_r, \D_{r}) \in \mtcl P_\b$, and $q \leq_\a
r|_\a$, then $$r \wedge_\a q:=(F_q\cup(F_r\restr [\a,\, \b)), \D_{q}\cup \D_{r})$$ is a condition in $\mtcl P_\b$ extending $r$. In particular, every maximal antichain in $\mtcl P_\a$ is a maximal antichain in $\mtcl P_\b$, and therefore $\mtcl P_\a$ is a complete suborder of $\mtcl P_\b$.
\end{lemma}

\begin{proof}

The proof proceeds by induction on $\b\geq\a$. The case $\beta = \alpha$ follows from Lemma \ref{preorder}, since it implies that $r \wedge_\a q$ is equivalent to $q$. Let us assume now that $\beta=\sigma+1$ with $\sigma \geq \alpha$. Clearly, $r \wedge_\a q$ satisfies clauses $(C0)$ and $(C1)$ in the definition of $\mtcl P_{\s+1}$. By the induction hypothesis we know that the restriction of $r \wedge_\a q$ to $\sigma$, that is, $$(r\wedge_\a q)|_\s= (F_q\cup(F_r\restr [\a,\, \s)), \D_{q}\cup \D_{r\av_\s})$$ is a condition in $\mtcl P_\s$ extending $r\av_\s$. Therefore, $r \wedge_\a q$ also satisfies $(C2)$. If $\s\notin dom(F_r)$, then $r \wedge_\a q$ is a condition in $\mtcl P_{\s+1}$, since clause $(C4)$ is automatically satisfied. If $\s\in dom(F_r)$, then $(r\wedge_\a q)\av_\s$ forces in $\mtcl P_\s$ that $F_r(\s)$ is in $\Upsilon(\s)$ (since $r\av_\s$ forces this and $(r\wedge_\a q)\av_\s$ extends $r\av_\s$). This concludes the verification of $(C3)$ for $q\wedge_\a r$. Now we check that $(r \wedge_\a q)\av_\s$ forces that $F_r(\s)$ is $(N[\dot G_\s], \Upsilon(\s))$--generic for all those $N$ in $\mathfrak N_{\s+ 1}$ such that $(N, \sigma+1) \in \D_{q}\cup \D_{r}$. But such an $N$ is such that $(N, \sigma+1) \in \D_{r}$ (by definition of $\mtcl P_\a$,  $\range(\D_q)\subseteq \alpha+1$). Since $r$ satisfies $(C4)$, $r\av_\s$ (and hence, the restriction of $r \wedge_\a q$ to $\sigma$) forces that $F_r(\s)$ is $(N[\dot G_\s], \Upsilon(\s))$--generic. Finally note that the induction hypothesis
and the inclusion $\D_r \subseteq \D_{r \wedge_\a q}$ together imply that $r \wedge_\a q$ extends $r$. The case when $\beta$ is a nonzero limit ordinal follows directly from the induction hypothesis.

\end{proof}

The following four amalgamation results are, essentially, Lemmas 3.5--3.8 from \cite{ASP}. The proofs of those lemmas are easy mechanical verifications, not particularly illuminating, and translate almost word by word to proofs of the present lemmas, so the familiarized reader may skip the present proofs.

\begin{lemma}\label{amalg1}
Let $\a<\k$ and let $q_0=(F_0, \, \D_0)$ and $q_1=(F_1, \, \D_1)$ be conditions in $\mtcl P_{\a+1}$ such that there is $v \in H(\kappa)$, a condition $r=(F_r, \D_r)$ in $\mtcl P_{\a}$, and a finite set $\{M_j : j \in  n\}$ with the following properties:
\begin{itemize}
\it[(a)] $\alpha+1 \leq \textsf{sup}(M_j \cap \k)$ and $(M_j,\a) \in \D_r$ for all $j < n$,
\it[(b)] $r$ extends both $q_0\av_\a$ and $q_1\av_\a$,
\it[(c)] $\a\in \dom(F_0)\cap \dom(F_1)$ and $r$ forces in $\mtcl P_\a$ that $v$ extends both $F_0(\a)$ and $F_1(\a)$ in $\Upsilon(\alpha)$, and
\it[(d)] $r$ forces in $\mtcl P_\a$ that $v$ is $(M_j[G_\a], \Upsilon(\a))$--generic for all $j < n$ such that $M_j \in \mathfrak N_{\a+1}$.

\end{itemize}
Then, $$q_2=(F_r\cup\{\la \a, v \ra\}, \, \D_{r} \cup \D_{0} \cup \D_{1} \cup \{(M_j, \a+1) : j \in  n\})$$ is a condition in $\mtcl P_{\a+1}$ extending both $q_0$ and $q_1$.
\end{lemma}

\begin{proof}
First we check that $q_2$ is in $\mtcl P_{\a+1}$. It follows from $(a)$ and $(b)$ together with Lemma \ref{preorder} that the restriction of $q_2$ to $\a$ is equivalent to $r$, and hence that $q_2$ satisfies clauses $(C0)$--$(C2)$. Condition $(C3)$ for $q_2$ follows from (c).
Finally we must show that $r$ forces that $v$ is $(N[\dot G_\a], \Upsilon(\a))$--generic for all those $N$ in $\mathfrak N_{\a+ 1}$ such that $(N, \a+1)$ is in $\D_{0}\cup \D_{1} \cup \{(M_j, \a+1) : j \in  n\}$ (recall that if $(N,\g) \in \D_r$, then $\gamma \leq \alpha$). But for such an $N$, if $(N, \a+1) \in \D_{i}$ ($i \in \{0, 1\}$), it suffices to recall that $r \leq_\a q_i\av_\a$, that $r$ forces that $v$ extends both $F_0(\a)$ and $F_1(\a)$, and that (by clause $(C4)$ applied to $q_i$) $q_i\av_\a$ forces that $F_i(\a)$ is $(N[\dot G_\a], \Upsilon(\a))$--generic. Hence, $r$ forces that $v$ is $(N[\dot G_\a], \Upsilon(\a))$--generic. The case when $(N, \a+1) \in \{(M_j, \a+1) : j \in  n\}$ follows from $(d)$. Finally note that $(b)$, $(c)$ and the inclusion $\D_i \subseteq \D_{0} \cup \D_{1} \cup \{(M_j, \a+1) : j \in  n\}$ imply together that $q_2$ extends $q_i$ for $i\in\{0, 1\}$.
\end{proof}

Exactly the same proof establishes the following variant of Lemma \ref{amalg1}.

\begin{lemma}\label{amalg1.5}
Let $\a<\k$ and let $q_0=(F_0, \, \D_0)$ and $q_1=(F_1, \, \D_1)$ be conditions in $\mtcl P_{\a+1}$, $r=(F_r, \D_r)$ a condition in $\mtcl P_{\a}$, and $\{M_j : j \in  n\}$ a finite set with the following properties:
\begin{itemize}
\it[(a)] $\alpha+1 \leq \textsf{sup}(M_j \cap \k)$ and $(M_j,\a) \in \D_r$ for all $j < n$,
\it[(b)] $r$ extends both $q_0\av_\a$ and $q_1\av_\a$, and
\it[(c)] $\a\notin \dom(F_0)\cup \dom(F_1)$.
\end{itemize}
Then, $$q_2=(F_r, \, \D_{r} \cup \D_{0} \cup \D_{1} \cup \{(M_j, \a+1) : j \in  n\})$$ is a condition in $\mtcl P_{\a+1}$ extending both $q_0$ and $q_1$.

Suppose, in addition, that $v \in H(\kappa)$ is such that

\begin{itemize}

\it[(d)] $r$ forces in $\mtcl P_\a$ that $v$ is $(M_j[G_\a], \Upsilon(\a))$--generic for all $j < n$ such that $M_j \in \mathfrak N_{\a+1}$, and

\it[(e)] $r$ forces in $\mtcl P_\a$ that $v$ is $(N[G_\a], \Upsilon(\a))$--generic for all $N$ such that $(N, \a+1)\in \D_0\cup\D_1$ and $N \in \mathfrak N_{\a+1}$.

\end{itemize}

Then, $$q'_2=(F_r\cup\{\la \a, v\ra\}, \, \D_{r} \cup \D_{0} \cup \D_{1} \cup \{(M_j, \a+1) : j \in  n\})$$ is a condition in $\mtcl P_{\a+1}$ extending both $q_0$ and $q_1$.
\end{lemma}

\begin{lemma}\label{amalg2}
Assume that $0\leq \sigma < \alpha \leq \kappa$. Let $q_\x=(F_\x,\D_\x)$ $(\x \in \{0, 1\})$ be  conditions in $\mtcl P_\a$ such that $\supp(q_0)\cup \supp(q_1) \subseteq \sigma$ and such that there exists a condition $r =(F_r, \D_r) \in \mtcl P_{\s}$ extending both $q_0\av_\s$ and $q_1\av_\s$. Then $q_0$ and $q_1$ are compatible in $\mtcl P_\a$.
\end{lemma}

\begin{proof}

Define $q_2=(F_r, \D_r\cup \D_{0} \cup \D_{1})$. We prove by induction on $\b$, $\s\leq \b\leq \a$, that $q_2\av_\b$ is a condition in $\mtcl P_\b$ extending $q_0\av_\b$ and $q_1\av_\b$. The successor step follows from Lemma \ref{amalg1.5}.
\end{proof}

\begin{lemma}\label{amalg3}
Given $\b\leq\k$ and given conditions $q_\x=(F_\x,\D_\x)$ ( for $\x \in \{ 0,1 \}$) in $\mtcl P_\b$, let $Z_{\x} = \supp(q_\x)\cup(\b \cap \bigcup \dom(\D_{q_\x}))$. Let $\a\leq \b$ be such that $Z_{0} \cap Z_{1} \subseteq \a$, and assume there is a condition $r=(F_r, \D_r)$ in $\mtcl P_\a$ extending $q_0 \av_\a$ and $q_1 \av_\a$. Let $F_r^{0,1} = F_r\cup (F_0\restr[\a,\,\b))\cup (F_1\restr[\a,\,\b))$.

Then the natural amalgamation of $r$, $q_0$ and $q_1$, i.e.,
$$(q_0 \wedge q_1) \wedge_\a r:= (F_r^{0,1}, \D_r \cup \D_0\cup \D_1)$$ is a $\mtcl P_\b$--condition extending $q_0 $ and $q_1$.
\end{lemma}

\begin{proof}
The proof is by induction on $\beta\geq \alpha$. The case $\beta = \alpha$ follows from Lemma \ref{preorder}, since this lemma implies that  $(q_0 \wedge q_1) \wedge_\a r$ is equivalent to $r$. We consider now the case when $\beta=\sigma+1$ with $\sigma \geq \alpha$. Clearly, $(q_0 \wedge q_1) \wedge_\a r$ satisfies clauses $(C0)$ and $(C1)$. Using the induction hypothesis we know that the restriction of the amalgamation to $\sigma$ is a condition in
$\mtcl P_\s$ which extends both $q_0\av_\s$ and $q_1\av_\s$. In particular, if $\s\in dom(F^{0, 1}_r)$, then $((q_0 \wedge q_1) \wedge_\a r)\av_\s$ forces that $F^{0, 1}_r(\s)$ is in $\Upsilon(\s)$. Therefore, $ (q_0 \wedge q_1) \wedge_\a r$ also satisfies $(C2)$ and $(C3)$.
Let us assume that $\sigma \in dom(F^{0, 1}_r)$. In fact, since $\supp(q_0) \cap \supp(q_1) \subseteq \alpha$ and $\sigma \geq \alpha$, we may assume that $\sigma \in supp(q_0) \setminus supp(q_1)$ (the proof when $\sigma \in supp(q_1)$ is identical). We must show that the condition $((q_0 \wedge q_1) \wedge_\a r)\av_\s$ forces that $F_r^{0,1}(\s)$ is $(N[\dot G_\s], \Upsilon(\s))$--generic for all those $N$ in $\mathfrak N_{\s+ 1}$ such that $(N, \sigma+1) \in \D_{0}\cup \D_{1}$. Since $\s \geq \alpha$ and $Z_{0} \cap Z_{1} \subseteq \a$, it follows for such an $N$ that $(N, \sigma+1) \in \D_0$. Since $q_0$ satisfies $(C4)$ and  $F_r^{0,1}(\sigma)= F_0(\sigma)$, $q_0\av_\s$ (and hence, $((q_0 \wedge q_1) \wedge_\a r)\av_\s$) forces what we want.  Finally note that the induction hypothesis, the choice of
$F_r^{0,1}$ and the inclusion $\D_0 \cup \D_1 \subseteq \D_{(q_0 \wedge q_1) \wedge_\a r}$ together imply that $(q_0 \wedge q_1) \wedge_\a r$ extends $q_0$ and $q_1$. The case when $\beta$ is a nonzero limit ordinal follows directly from the induction hypothesis.
\end{proof}

The following lemma gives a representation of $\mtcl P_{\a+1}$ as a certain dense subset of an iteration of the form $\mtcl P_\a\ast\dot{\mtcl Q}_\a$.

\begin{lemma}\label{repr}
For all $\a<\k$, $\mtcl P_{\a+1}$ is isomorphic to a dense suborder of $\mtcl P_\a\ast\dot{\mtcl Q}_\a$, where $\dot{\mtcl Q}_\a$ is, in $\V^{\mtcl P_\a}$, the collection of all pairs $(W, Q)$ such that $W$ has cardinality at most $1$ and

\begin{itemize}

\it[$(\circ)$] if $v \in W$, then there is some $r = (F_r, \D_r)\in\dot G_\a$ such that $$(F_r\cup\{\la\a, v\ra\}, \D_r\cup \{(N, \a+1)\,:\, N\in Q\})\in\mtcl P_{\a+1},$$

\it[$(\circ)$] if $W=\emptyset$, then there is some $r = (F_r, \D_r)\in\dot G_\a$ such that $$(F_r, \D_r\cup \{(N, \a+1)\,:\, N\in Q\})\in\mtcl P_{\a+1},$$
\end{itemize}

\noindent ordered by $(W_1,  Q_1)\leq_{\dot{\mtcl Q}_\a}(W_0, Q_0)$ if and only if

\begin{itemize}

\it[(i)] $Q_0\sub Q_1$, and

\it[(ii)] if $W_0=\{v\}$, then $W_1=\{v'\}$ for some $v'\in\Upsilon(\a)$ which extends $v$ in $\Upsilon(\a)$.
\end{itemize}
\end{lemma}

\begin{proof}
Let $\wdtld{\mtcl P}_{\a+1}$ consist of all $(r, \check x)$, where $r\in\mtcl P_\a$ and $r\Vdash_{\mtcl P_\a}\check x\in\dot{\mtcl Q}_\a$. Then $\psi:\mtcl P_{\a+1}\into\wdtld{\mtcl P}_{\a+1}$ is an isomorphism, where $\psi(q)=(q\av_\a,\,\check x)$, $x=(\{F_q(\alpha)\}, \D_q^{-1}(\a+1))$ if $\a\in\supp(q)$, and $x=(\emptyset, \D_q^{-1}(\a+1))$ if $\a\notin\supp(q)$.
\end{proof}

\begin{lemma}\label{Qchain}
If $\alpha$ and $\dot{\mtcl Q}_\a$ are as in the statement of Lemma \ref{repr}, then $\dot{\mtcl Q}_\a$ is forced by $\mtcl P_\a$ to have the $\aleph_2$--chain condition.
\end{lemma}

\begin{proof}
Let $G$ be $\mtcl P_\alpha$--generic filter over $\V$ and let $\mtcl Q_\a$ be the interpretation of $\dot{\mtcl Q}_\a$ by $G$.  Working in $\V[G]$ suppose $\{(W_i, Q_i) \,:\,i < \omega_2\}\subseteq\mtcl Q_\a$ and  $W_i=\{v_i\}$ is nonempty for each $i$. Since $\Upsilon(\alpha)$ is forced by $\mtcl P_\a$ to have the $\aleph_2$--c.c., there are $v \in \Upsilon(\a)$, $i_0$, $i_1<\o_2$, and $r \in G$ such that $r$ forces that $v$ extends both $v_{i_0}$ and $v_{i_1}$. The proof will be finished in this case once we show $(F_r\cup\{\la\a, v\ra\}, \D_r\cup \{(N, \a+1)\,:\, N\in Q_{i_0} \cup Q_{i_1}\})$ is a condition in $\mtcl P_{\a+1}$. Note that, by definition of $\mtcl Q_\a$, there are $r_0$ and $r_1$ in $G$ such that $(F_{r_0}\cup\{\la\a, v_{i_0} \ra\}, \D_{r_0}\cup \{(N, \a+1)\,:\, N\in Q_{i_0}\})$ and $(F_{r_1}\cup\{\la\a, v_{i_1} \ra\}, \D_{r_1}\cup \{(N, \a+1)\,:\, N\in Q_{i_1}\})$ are conditions in $\mtcl P_{\a+1}$. Since $G$ is a filter and by extending $r$ if necessary, we can assume that $r$ extends both $r_0$ and $r_1$. Now we are done by a  simple application of Lemma \ref{amalg1}. 

The proof for the case when $\{(W_i, Q_i) \,:\,i < \omega_2\}\subseteq\mtcl Q_\a$ is such that $W_i=\emptyset$ for cofinally many  $i$ in $\omega_2$ is very similar (the only difference is that we use the first part of Lemma \ref{amalg1.5} instead of Lemma \ref{amalg1}).
\end{proof}

The next step in the proof of Theorem \ref{mainthm} will be to show that all $\mtcl P_\a$ (for $\a\leq \k$) have the $\al_2$--chain condition. We first prove that $\mtcl P_0$ is $\al_2$--Knaster.\footnote{Given a cardinal $\m$, a partial order $\mtbb P$ is $\m$--Knaster if for every $X\in[\mtbb P]^\m$ there is $Y\in [X]^\m$ consisting of pairwise compatible conditions.}

\begin{lemma}\label{knaster} $\mtcl P_0$ is $\al_2$--Knaster.
\end{lemma}

\begin{proof}
Suppose $m<\o$ and $q_\x=\{N^\x_i\,:\,i<m\}$ is a $\mtcl P_0$--condition for each $\x<\o_2$. For notational convenience we are identifying a $\mtcl P_0$--condition $q$ with $\dom(\D_q)$, which is fine for this proof. By $\textsf{CH}$ we may assume that $\{\bigcup_{i<m}N^\x_i\,:\,\x<\o_2\}$ forms a $\D$--system with root $X$. Furthermore, by $\textsf{CH}$ we may assume, for all $\x$, $\x'<\o_2$, that the structures
$\la \bigcup_{i<m}N^\x_i,\in, \Phi,  X, N^\x_i\ra_{i<m}$ and $\la \bigcup_{i<m}N^{\x'}_i,\in, \Phi,  X, N^{\x'}_{i}\ra_{i<m}$  are isomorphic and that the corresponding isomorphism fixes $X$. The first assertion follows from the fact that there are only $\al_1$--many isomorphism types for such structures. For the second assertion note that, if $\Psi$ is  the unique isomorphism between $\la \bigcup_{i<m}N^\x_i,\in, \Phi,  X, N^\x_i\ra_{i<m}$ and $\la \bigcup_{i<m}N^{\x'}_i,\in,  \Phi, X, N^{\x'}_{i}\ra_{i<m}$, then the restriction of $\Psi$ to $X\cap\k$ has to be the identity on $X\cap\k$. Since  there is  a bijection $\Phi:H(\k)\into \k$ definable in $(H(\k), \in, \Phi)$, we have that $\Psi$ fixes $X$ if and only if it fixes $X\cap\k$. It follows that $\Psi$ fixes $X$.
Hence, by Lemma \ref{iso3} we have, for all $\x$, $\x'<\o_2$, that $q_\x\cup q_{\x'}$ extends both $q_\x$ and $q_{\x'}$.
\end{proof}

\begin{lemma}\label{cc}
For every $\a\leq \k$, $\mtcl P_\a$ has the $\al_2$--chain condition.
\end{lemma}

\begin{proof}
The proof is by induction on $\a$, being the conclusion for $\a=0$ a consequence of the above lemma.

For $\a=\s+1$ the conclusion follows immediately from Lemmas \ref{repr} and \ref{Qchain} together with the induction hypothesis for $\s$ and the fact that the $\al_2$--c.c.\ is preserved under two--step iterations.

For $\a$ a nonzero limit ordinal, suppose $q_\x$ is a $\mtcl P_\a$--condition for all $\x<\o_2$. Suppose first that $\cf(\a) \neq \o_2$. There is then some $\s<\a$ such that $I=\{\x<\o_2\,:\,\supp(q_\x)\sub\s\}$ has size $\al_2$. By induction hypothesis we may find distinct $\x$ and $\x'$ in $I$ such that $q_\x \av_\s$ and $q_{\x'} \av_\s$ are compatible in $\mtcl P_\s$. But now it follows from Lemma \ref{amalg2} that $q_\x$ and $q_{\x'}$ are compatible.

Finally, suppose $\cf(\a)=\o_2$. For each $\x < \o_2$, let $Z_{\x}$ be the union of the sets $\supp(q_\x)$ and $\a \cap \bigcup \dom(\D_{q_\x})$. By $\textsf{CH}$ we may find $I\sub\o_2$ of size $\al_2$ such that $\{Z_{\x} \,:\,\x\in I\}$ forms a $\D$--system with root $X$.

Let now $\s<\a$ be such that $X\sub \s$ ($\s$ exists by $\cf(\a)\geq\o_1$). By induction hypothesis we may find distinct $\x$ and $\x'$ in $I$ such that $q_\x \av_\s$ and $q_{\x'} \av_\s$ are compatible in $\mtcl P_\s$. Let $r_{\x, \x'}$ be a condition in $\mtcl P_\s$ extending $q_\x\av_\s$ and $q_{\x'}\av_\s$. By Lemma \ref{amalg3} it follows that the natural amalgamation of $r_{\x, \x'}$, $q_\x$ and $q_{\x'}$ is a $\mtcl P_\a$--condition extending $q_\x$ and $q_{\x'}$.

\end{proof}

\begin{corollary}\label{cor00} For every $\a\leq\k$ and every cardinal $\k'\geq\k$, $\mtcl P_\a$ forces $H(\k')^{\V[\dot G_\a]}=H(\k')^{\V}[\dot G_\a]$ and forces $(N^\ast\cap H(\k'))[\dot G_\a]= N^\ast[\dot G_\a]\cap H(\k')^{\V[\dot G_\a]}$ whenever $\t$ is regular and $N^\ast$ is a countable elementary substructure of $H(\t)$ such that $\mtcl P_\a\in N^\ast$ and $\k'\in N^\ast$.\end{corollary}

\begin{definition}
Given $\a\leq \kappa$, a condition $q\in\mtcl P_\a$, and a countable elementary substructure $N \prec H(\kappa)$, we will say that $q$ is $(N,\, \mtcl P_\a)$--pre-generic in case
\begin{itemize}

\it[$(\circ)$] $\a < \k$ and the pair $(N, \a)$ is in $\Delta_q$, or else

\it[$(\circ)$] $\a = \k$ and the pair $(N, \textsf{sup}(N\cap \k))$ is in $\Delta_q$.

\end{itemize}
\end{definition}

The properness of all $\mtcl P_\a$ is an immediate consequence of Lemma \ref{horribilis}. Much of the machinery in our definition of $\la\mtcl P_\a\,:\,\a\leq\k\ra$ is there precisely to make the proof of the lemma work. This lemma is proved by reciting very much the same mantra as in the proof of a corresponding lemma in \cite{ASP}.  We will repeat this mantra for the reader's benefit, of course adapted to the present situation.

\begin{lemma}\label{horribilis}

Suppose $\a\leq \kappa$ and $N^\ast\in\mathfrak N^\ast_\a$. Let $N=N^\ast\cap H(\k)$. Then the following conditions hold.

\begin{itemize}

\it[$(1)_\a$] For every $q\in N$ there is $q'\leq_\a q$ such that $q'$ is $(N,\, \mtcl P_\a)$--pre-generic.

\it[$(2)_\a$] If $\mtcl P_\a\in N^\ast$ and $q\in\mtcl P_\a$ is $(N,\, \mtcl P_\a)$--pre-generic, then $q$ is $(N^\ast,\, \mtcl P_\a)$--generic.

\end{itemize}

\end{lemma}

\begin{proof}
The proof will be by induction on $\a$. We start with the case $\a =0$. For simplicity we are going to identify a $\mtcl P_0$--condition $q=(\emptyset, \D_q)$ with $\dom(\D_q)$. The proof of $(1)_0$ is trivial: It suffices to set $q'= q\cup\{N\}$.

The proof of $(2)_0$ is also easy: Let $E$ be a dense subset of $\mtcl P_0$ in $N^\ast$. It suffices to show that there is some condition in $E\cap N^\ast$ compatible with $q$. Notice that $q\cap N^\ast\in\mtcl P_0$ by Lemma \ref{iso2} (i). Hence, we may find a condition $q^\circ\in E\cap N^\ast$ extending $q\cap N^\ast$. Now, by Lemma \ref{iso2} there is a partial $\Phi$--symmetric system $\mtcl M$ extending $q\cup q^\circ$. It follows that $\mtcl M$ is a condition in $\mtcl P_0$ extending $q$ and $q^\circ$.

Let us proceed to the case $\a=\s+1$. We start by proving $(1)_\a$. Assume first that $\sigma\in\supp(q)$.  By $(1)_\s$ we may assume, by extending  $q\av_\s$, that $q\av_\s$ is $(N, \mtcl P_\s)$--pre-generic. Let $\dot D$ be the $\lhd$--first $\mtcl P_\s$--name for a club $D$ of $[H(\k)]^{\al_0}$ in $\V$ such that $D$ witnesses that $\Upsilon(\s)$ is $\aleph_{1.5}$--c.c.\ relative to $\V$ and $\dot G_\s$. Note that $q\av_\s$ forces $N \in \dot D$ since $\dot D\in N^\ast$ and since $q\av_\s$ forces $H(\k)^{N^\ast[\dot G_\s]}=N$ by $(2)_\s$.
But then there is some $t \in \mtcl P_\s$ extending $q\av_\s$ and some $y\in H(\k)$ such that $t$ forces that $y$ is an $(N[\dot G_{\s}], \Upsilon(\s))$--generic condition extending  $F_{q}(\s)$. It suffices to define $q'$ as the condition $(F_t\cup\{\la\s,  y \ra\}, \D_q \cup \D_t \cup \{(N,\a)\})$ (Lemma \ref{amalg1} ensures that $q'$ is indeed a condition extending $q$).

The proof in the case that $q=(F, \D)$ with $\dom(F)\sub\s$ can be reduced to the previous case by the following Claim.

\begin{claim}\label{quecosa}
If $q=(F, \D)$ and $\s\notin \dom(F)$, then we can find a condition $q'=(F', \D')$ extending $q$ and such that $\s\in \dom(F')$.
\end{claim}

\begin{proof}
This is true, using Lemma \ref{amalg1.5}, by essentially the same argument as above since $(2)_\sigma$ guarantees that $q\av_\s$ is also $(M^\ast,\, \mtcl P_\s)$--generic for all $M^\ast \in \mathfrak N^\ast_{\s+1}$ such that $(M, \s+1)\in\D_q$ for $M=M^\ast\cap H(\k)$, which implies that a condition forcing  that all these $M$ are in $\dot D$ can be found as in that argument. Also note that, by its being definable, the  weakest condition of $\Upsilon(\s)$ is forced to be in $H(\k)^{M^\ast[\dot G_\s]}= M[\dot G_\s]$  for all these $M$'s, where the equality holds by Corollary \ref{cor00}.
\end{proof}

Now let us prove $(2)_\a$.
Let $A$ be a maximal antichain of $\mtcl P_\a$ in $N^\ast$,
and assume without loss of generality that $q=(F_q, \D_q)$ extends some condition $q^\ast$ in $A$. We must show $q^\ast\in N$.
Note that $A$ is in $N$ by the $\al_2$--c.c.\ of $\mtcl P_\a$. By Claim \ref{quecosa} we may also assume that $\sigma \in \dom(F_q)$.  Let $G_{\sigma}$ be a $\mtcl P_{\sigma}$-generic filter over $\V$ with $q\av_{\sigma} \in G_{\sigma}$. By $(2)_\s$ we have that $G_{\sigma}$ is also generic over
$N^\ast$.

Let $E$ be the set of $\Upsilon(\s)$--conditions $v$ such that either

\begin{itemize}

\it[(i)] there exists some $a  \in \mtcl P_\a$ extending some member of $A$ such that $a\av_\s\in  G_\s$, $\s\in \dom(F_a)$, and such that $F_a(\s)=v$, or else

\it[(ii)] there is no $a \in \mtcl P_\a$ extending any member of $A$ such that $a\av_\s\in \dot G_\s$, $\s\in \dom(F_a)$, and such that $F_a(\s)\leq_{\Phi(\s)} v$.

\end{itemize}

$E$ is a dense subset of $\Upsilon(\s)$, and $E\in N^\ast[G_\s]$ by $E$ being definable in $H(\k^+)^{\V}[\dot G_\s]$ from $\mtcl P_\a$ together with the fact that $H(\k^+)^{N^\ast[G_\s]}\elsub H(\k^+)^{\V}[\dot G_\s]$ and $N^\ast$
contains $\mtcl P_\a$ and $A$. Note that $E$ is in fact in $N[\dot G_\s]$ by Corollary \ref{cor00}. Suppose $F_q(\s)=\ov v$. Since $\ov v$ is $(N[ G_\s],\,\Upsilon(\s))$--generic, we may find some $v'\in E\cap N[ G_\s]$ and some $v^\ast$ $\Upsilon(\s)$--extending both $v'$ and
$\ov v$.

\begin{claim}\label{cl}
Condition (i) above holds for $v'$.
\end{claim}

\begin{proof}
Let $r$ be a condition in $G_\s$ extending $q\av_\s$ and deciding $v^\ast$, and let $u = (F_r\cup\{\la\s, v^\ast\ra\}, \D_r\cup \D_q)$. By Lemma \ref{amalg1}, $u$ is a $\mtcl P_\a$--condition extending $q$. In particular, $u$ extends a condition in $A$, and therefore it witnesses the negation of condition (ii) for $v'$, so condition (i) must hold for $v'$.

\end{proof}

By the above claim and by $H(\k^+)^{N^\ast[G_\s]}\elsub H(\k^+)^{\V}[G_\s]$ there is $a$ in $N^\ast[G_\s]$ witnessing that condition (i) holds for $v'$, and actually $a\in N$ since $N^\ast[G_\s]\cap \V = N^\ast$ by $(2)_\s$. Now we extend $q\av_\s$ to a condition $r$ deciding $a$, and deciding also some common extension $v^\ast\in\Upsilon(\s)$ of $\ov v$ and $v'$. We may also assume that $r$ extends $a\av_\s$.  It follows that $(F_r \cup\{\la\s, v^\ast\ra\}, \D_r\cup \D_a\cup\D_q)$ is a common extension of $q$ and $a$ by Lemma \ref{amalg1}. Hence $q^\ast=a$.

It remains to prove the lemma for the case when $\a$ is a nonzero limit ordinal. The proof of $(1)_\a$ is easy.
Let $\s\in N\cap\a$ be above $\supp(q)$. By induction hypothesis we may find $r\in\mtcl P_\s$ extending $q\av_\s$ and such that $(N, \s)\in\D_r$. One can now easily check that the result of stretching the marker $\s$ in $(N, \s)$ up to $\a$ if $\a < \k$ and up to $\textsf{sup}(N\cap\k)$ if $\a = \k$ is a condition in $\mtcl P_\a$ extending $q$ with the desired property.

For $(2)_\a$, let $A$ be a maximal antichain of $\mtcl P_\a$ in $N^\ast$,
and assume without loss of generality that $q=(F_q, \D_q)$ extends some condition $q^\ast$ in $A$. We must show $q^\ast\in N$.
Suppose first that $\cf(\a)=\o$. In this case we may take $\s\in N^\ast\cap\a$ above $\supp(q)$. Let $G_\s$ be $\mtcl P_\s$--generic with $q\av_\s\in G_\s$. In $N^\ast[G_\s]$ it is true that there is a condition $q^\circ\in\mtcl P_\a$ such that

\begin{itemize}

\it[(a)] $q^\circ\in A$ and $q^\circ\av_\s\in G_\s$, and

\it[(b)] $\supp(q^\circ)\sub \s$.

\end{itemize}

\noindent (the existence of such a $q^\circ$ is witnessed in $\V[G_\s]$ by $q^\ast$.)

Since $q\av_\s$ is $(N^\ast, \mtcl P_\s)$--generic by induction hypothesis, $q^\circ\in N^\ast$. By extending $q$ below $\s$ if necessary, we may assume that $q\av_\s$ decides $q^\circ$ and extends $q^\circ\av_\s$. But now, if $q=(F_q, \D_q)$, the natural amalgamation $(F_q, \D_q\cup \D_{q^\circ})$ of $q$ and $q^\circ$ is a $\mtcl P_\a$--condition extending them by Lemma \ref{amalg2}. It follows that $q^\ast = q^\circ$.

Finally, suppose $\cf(\a)\geq\o_1$. This will be the only place where we use the (partial) symmetry of $\dom(\D_q)$. The crucial observation in this case is that if $N'\in \dom(\D_q)$ is such that $\d_{N'}<\d_N$ and $N''\in N\cap\dom(\D_q)$ is such that $\d_{N''}=\d_{N'}$ and $N''\cap N'=N'\cap N$, then $\textsf{sup}(N'\cap N\cap \a)\leq \textsf{sup}(N''\cap\a)$. Hence, since such an $N''$ is countable, we may fix $\s\in N\cap \a$ above $\supp(q)\cap N$ and above $\textsf{sup}(N'\cap N\cap \a)$ for all $N'\in \dom(\D_q)$ with $\d_{N'}<\d_N$.

As in the above case, if $G_\s$ is $\mtcl P_\s$--generic with $q\av_\s\in G_\s$, then in $N^\ast[G_\s]$ we can find a condition $q^\circ \in A$ such that $q^\circ\av_\s\in G_\s$ (again, the existence of such a condition is witnessed in $\V[G_\s]$ by $q$), and such a $q^\circ$ will necessarily be in $N^\ast$. By extending $q$ below $\s$ we may assume that $q\av_\s$ decides $q^\circ$ and extends $q^\circ\av_\s$. The proof of $(2)_\a$ in this case will be finished if we can show that there is a condition $\ov{q}$ extending $q$ and $q^\circ$.

The condition $\ov{q}$ can be built by recursion on $\supp(q^\circ)\setminus\s$ (note that by the choice of $\sigma$, $\min(\supp(q)\setminus \sigma)\geq \textsf{sup}(N\cap\alpha)$, and therefore $\min(\supp(q)\setminus \sigma)> \max(\supp(q^\circ))$). This finite construction mimics the proof of $(1)_\a$ for successor $\a$. The details are as follows.

Let $(\x_i)_{i<r}$ be the strictly increasing enumeration of $\supp(q^\circ)\setminus \s$. We may assume that $r>0$, and therefore $r-1\geq 0$. We build a sequence $(q_i)_{i<r}$ of conditions as follows:

For $i=0$, we first extend $q\av_\s$ to a $\mtcl P_{\x_0}$--condition $r$ extending $q\av_{\x_0}$ and $q^\circ\av_{\x_0}$. $r$ can be found by appealing to Lemma \ref{amalg2} if $\s < \x_0$, and if $\s=\x_0$ it is enough of course to set $r=q\av_\s$.

If $\ov{N}$ is a relevant structure -- meaning that $(\ov{N}, \x_0+1)\in\D_{q\av_{\x_0+1}}$ and $\ov{N}\in\mathfrak N_{\x_0+1}$ --, then $r$
forces also that the first club $D\sub [H(\k)^{\V}]^{\al_0}$ in $\V$ (in the well--order of $H(\k^+)[\dot G_{\x_0}]$ induced by $\lhd$) witnessing that the poset $\Upsilon(\x_0)$ is $\aleph_{1.5}$--c.c.\ relative to $\V$ and $\dot G_{\x_0}$ is such that every relevant $\ov{N}$ is in $D$. (For such an $\ov{N}$ there is some $\ov{N}^\ast \in \mathfrak N^\ast_{\x_0+1}$ containing $\lhd$ and such that $\ov{N}=\ov{N}^\ast \cap H(\kappa)$ which, since $\lhd\in \ov{N}^\ast$, implies that $\ov{N}^\ast$ contains a name $\dot D$ for that club. Applying this fact and $(2)_{\x_0}$ we conclude that $\overline{q}\av_{\x_0}$ forces $\ov{N} \in \dot D$.) Finally, $r$ forces that $F_{q^\circ}(\x_0)$ is in $N[\dot G_{\x_0}]$ and hence in $N'[\dot G_{\x_0}]$ for some relevant $N'$ with $\d_{N'[\dot G_{\x_0}]}$ minimal. The reason is that $\d_{N'[\dot G_{\x_0}]}=\d_{N'}$ for every relevant $N'$ by $(2)_{\x_0}$ and that by our choice of $\s$ we know that
if $\ov{N}\in \dom(\D_{q})$ and $\x_0 \in \ov{N}$, then $\d_{\ov{N}}\geq\d_{N}$. Putting these facts together we get that there is some $v^\ast\in H(\k)$ and some extension of $r$ forcing that $v^{\ast}$ $\Upsilon(\x_0)$--extends $F_{q^\circ}(\x_0)$ and is $(\ov{N}[\dot G_{\x_0}],\,\Upsilon(\x_0))$--generic for all relevant $\ov{N}$.
It follows now from Lemma \ref{amalg1} that there is a $\mtcl P_{\x_0+1}$--condition $q_0$ extending $r$, $q\av_{\x_0+1}$ and $q^\circ\av_{\x_0+1}$.

For $i$ such that $i+1<r$, we assume inductively that $q_i\in\mtcl P_{\x_i+1}$ extends $q\av_{\x_i+1}$ and $q^\circ\av_{\x_i+1}$, and obtain $q_{i+1}\in\mtcl P_{\x_{i+1}+1}$ from $q_i$ by arguing exactly as in the case $i=0$ with $\x_{i+1}$ instead of $\x_0$ and starting with $q_i$ rather than $q\av_\s$. In the end we obtain $q_{i+1}\in\mtcl P_{\x_{i+1}+1}$ extending both $q\av_{\x_{i+1}+1}$ and $q^\circ\av_{\x_{i+1}+1}$. Let $\mu = \xi_{r-1}=  \max(\supp(q^\circ))$ and let $$\ov{q} = (F_{q_{r-1}}\cup (F_q \restr [\mu+1,\, \a)), \D_{q_{r-1}} \cup \D_{q^\circ} \cup \D_{q})$$

\begin{claim}
$\ov q$ is a condition in $\mtcl P_{\a}$ extending both $q$ and $q^{\circ}$.
\end{claim}

\begin{proof}
We prove by induction that if $\mu+1 \leq \xi \leq \alpha$, then $\ov{q}\av_{\xi}$ is in $\mtcl P_{\xi}$ and $\ov{q}\av_{\xi} \leq_{\xi} q^\circ|_{\xi}$,  $q\av_{\xi}$. The case $\xi= \mu+1$ follows from Lemma \ref{preorder}, since it implies that $\ov{q}\av_{\mu+1}$ is equivalent to $q_{r-1}$ (recall that $q_{r-1} \leq_{\mu+1} q^\circ\av_{\mu+1}$, $q\av_{\mu+1}$). Assume now that $\xi=\eta+1$ for $\eta \geq \mu+1$. It suffices to show that $\ov{q}\av_{\eta+1}$ satisfies condition (C4) in the definition of $\mtcl P_{\eta+1}$ since it clearly satisfies the other conditions. In other words, we must show that if $\eta\in dom(F_q)$, then $\ov{q}\av_\eta$ forces that $F_q(\eta)$ is $(M[\dot{G}_\eta], \Upsilon(\eta))$--generic for all those $M \in \mathfrak N_{\eta+1}$ for which there exists an ordinal $\beta \geq \eta+1$ such that $(M, \beta) \in \D_{q_{r-1}} \cup \D_{q^\circ} \cup \D_{q}$. But such a pair $(M, \beta)$ cannot be in $\D_{q_{r-1}}$, since all markers occurring in side conditions in $q_{r-1} \in \mtcl P_{\mu+1}$ are at most $\mu+1< \eta+1$. On the other hand, we know that $\eta \in \supp(q) \setminus \sigma =\supp(q) \setminus (N\cap \alpha)$. It follows that there is no $M \in \dom(\D_{q^\circ})$ such that $M \in \mathfrak N_{\eta+1}$ (such a countable $M$ is in $N$, and therefore $M\cap \alpha\sub N\cap \alpha$), and hence $(M, \beta)$ is neither in $\D_{q^\circ}$. We conclude that such a pair $(M, \beta)$ is in $\D_{q}$. By (C4) applied to $q\av_{\eta+1}$, we have that $q\av_{\eta}$ (and hence, $\ov{q}\av_\eta$) forces what we want. Finally note that the induction hypothesis
$\ov{q}\av_{\eta} \leq_{\eta} q^\circ\av_{\eta}$, $q\av_{\eta}$, the definition of $\ov{q}$, and the fact that the maximum of the support of $q^{\circ}$ is $\mu< \eta$ together imply that $\ov{q}\av_{\eta+1} \leq_{\eta+1} q^\circ \av_{\eta+1}$, $q\av_{\eta+1}$. The case when $\xi$ is limit follows from the induction hypothesis.
\end{proof}

The above claim finishes the proof of $(2)_\a$ for limit $\a$ and the proof of the lemma.
\end{proof}

\begin{corollary} For all $\a\leq\k$, $\mtcl P_\a$ is proper.\end{corollary}

The following lemma is a straightforward application of Claim \ref{quecosa} and Lemma \ref{compll}.

\begin{lemma}\label{genh} For every $\a<\k$ and every condition $q\in\mtcl P_\k$, $q$ forces that the collection of all $y$ such that there is some $(F, \D)\in\dot G_\k$ with $F(\a)= y$ generates a $\V[\dot G_\a]$--generic filter on $\Upsilon(\a)$.\end{lemma}

Lemma \ref{bound} follows from the usual counting of nice names for subsets of $\k$ using $(\k^{<\k})^{\V}=\k$ and Lemma \ref{cc}.

\begin{lemma}\label{bound} $\mtcl P_\k$ forces $\k^{<\k}=\k$.\end{lemma}

\begin{lemma}\label{v-proper}
If $\dot{\mtcl R}$ is a $\mtcl P_\k$--name for an $\aleph_{1.5}$--c.c.\ poset included in $H(\kappa)^{\V}$, then $\mtcl P_\k$ forces that $\dot{\mtcl R}$ is $\aleph_{1.5}$--c.c.\ relative to $\V$ and $\dot G_\k$.

\end{lemma}

\begin{proof}
Let $G$ be generic for $\mtcl P_\k$, let $\mtcl R$ be the interpretation of $\dot{\mtcl R}$ by $G$, and let $\dot{\mtcl D}$ be a $\mtcl P_\k$--name for a club of $[H(\k^+)[\dot G_\k]]^{\aleph_0}$ witnessing that $\mtcl R$ has the $\aleph_{1.5}$--chain condition.\footnote{Note that there is such a $\dot{\mtcl D}$ since $\av\mtcl R\av^{\V[G]}\leq\k$  by Lemma \ref{bound} and therefore, by definition of $\al_{1.5}$--c.c., there must be a club of $[H(\k^+)^{\V[G]}]^{\al_0}$ in $\V[G]$ witnessing that $\mtcl R$ is $\al_{1.5}$--c.c. Also, remember that $\mtcl P_\k$ forces $H(\k^+)[\dot G_\k]= H(\k^+)^{\V[\dot G_\k]}$ (Corollary \ref{cor00}).} Fix a sufficiently large cardinal $\theta$ such that $\dot{\mtcl R}$ and $\dot{\mtcl D}$ are in $H(\theta)$. Define $\mathfrak N^\ast$ as the set of all countable elementary substructures of $H(\theta)$ containing  all the relevant parameters. Let
also $\mathfrak N=\{N^\ast\cap H(\k)  \,:\,N^\ast\in\mathfrak N^\ast\}$ and note that, by properness together with Lemma \ref{horribilis}, the club $\mathfrak N$ witnesses that $\dot{\mtcl R}$ is $\aleph_{1.5}$--c.c.\ relative to $\V$ and $\dot G_\k$:

If $u \in \mtcl P_\kappa$ and $\{N_i\,:\, i \in m\} \subseteq \mathfrak N$ are such that $\D_u$ includes $\{(N_i, \textsf{sup}(N_i \cap\kappa) )\,: \,i<m\}$, then $u$ forces that $(N_i^\ast \cap H(\k^+))[\dot G_\k]=H(\k^+)^{N_i^\ast[\dot G_\k]}$ is in $\dot{\mtcl D}$ whenever $N_i^\ast \in \mathfrak N^\ast$ is such that $N_i^\ast= N_i \cap H(\kappa)$. Also note that, by the assumption $\dot{\mtcl R} \sub H(\k)$ and the $\aleph_2$--chain condition of this poset, $u$ forces that  $(N_i^\ast \cap H(\k^+))[\dot G_\k]$ and $N_i[\dot G_\k]$ have the same maximal antichains of $\dot{\mtcl R}$. It follows that $u$ forces that a condition in $\dot{\mtcl R}$ is
$((N_i^\ast \cap H(\k^+))[\dot G_\k], \dot{\mtcl R})$--generic iff it is $(N_i[\dot G_\k], \dot{\mtcl R})$--generic. Thus $u$ forces that if $y\in\dot{\mtcl R}$ belongs to some $N_i[\dot G_\k]$ with $\d_{N_i}$ minimal among $\{\d_{N_j}\,:\,j<m\}$, then there is an extension of $y$ in $\dot{\mtcl R}$ which is $(N_i[\dot G_\k],\,\dot{\mtcl R})$--generic for all $i$. This is true since $u$ forces $\d_{(N_i^\ast \cap H(\k^+))[\dot G_\k]}=\d_{N_i[\dot G_\k]}=\d_{N_i}$ by Lemma \ref{horribilis}.
\end{proof}

\begin{definition}\label{reflection}
Let $Q$ be an elementary substructure of $H(\t)$, for some regular $\t >\t_\k$, and suppose $Q$ is closed under $\o$--sequences and contains $\Phi$, $\lhd$ and $\vec X$. Suppose $Q\cap\k$ is an ordinal $\d$ in $\k$. Let $\a\leq \d$, let $q$ be a condition in $\mtcl P_\a$, and let $q^\ast\in\mtcl P_\a\cap Q$ be a condition with the following properties.

\begin{enumerate}
\it $F_{q^\ast}=F_q$
\it There is an enumeration $((N_i, \g_i)\,:\,i<n)$ of $\dom(\D_q)$ and an enumeration $((M_i, \tld\g_i)\,:\,i<n)$ of $\dom(\D_{q^\ast})$ such that
\begin{itemize}
\it there is a set $R$ with $$(\bigcup\dom(\D_q))\cap(\bigcup\dom(\D_{q^\ast}))=R=(\bigcup\dom(\D_q))\cap Q$$ and an isomorphism $\Psi$ between $$\la\bigcup(\dom(\D_{q^\ast}),\in, R, \Phi, M_i\ra_{i<n}$$ and $$\la\bigcup(\dom(\D_q)), \in, R, \Phi, N_i\ra_{i<n}$$ which is the identity on $R$, and such that
\it for all $i<n$,  $\tld\g_i=\g_i$ if $\g_i<\d$ and $\tld\g_i=\textsf{sup}(N_i\cap\d)$ if $\g_i=\d$, and for all $\x\in R\cap N_i\cap \a$,
\begin{enumerate}
\it $M_i\in\mathfrak N_{\x+1}$ iff $N_i\in\mathfrak N_{\x+1}$, and
\it if $M_i\in \mathfrak N_{\x+1}$, then $\Psi\restr M_i$ is an isomorphism between $(M_i, \in, \mtcl P_\x)$ and $(N_i, \in, \mtcl P_\x)$.
\end{enumerate}
\end{itemize}
\end{enumerate}

Then we say that $q^\ast$ is {\emph a reflection of $q$ to $Q$}.
\end{definition}

Lemma \ref{fa} will make use of the following reflection result.

\begin{lemma}\label{refl}
Let $Q$ be an elementary substructure of $H(\t)$, for some regular $\t >\t_\k$, and suppose $Q$ is closed under $\o$--sequences and contains $\Phi$, $\lhd$ and $\vec X$. Suppose $Q\cap\k$ is an ordinal $\d$ in $\k$.  Let $\a\leq\d$, let $q_0\in\mtcl P_\a\cap Q$ and let $q=(F_q, \ \{(N_i, \g_i)\,:\,i<n\})$ be a $\mtcl P_\a$--condition extending $q_0$. Then

\begin{enumerate}
\it there is a reflection of $q$ to $Q$ extending $q_0$, and
\it for every $q^\ast$, if $q^\ast=(F_q, \{(M_i, \tld\g_i)\})$ is a reflection of $q$ to $Q$ as witnessed by the isomorphism $$\Psi:\la\bigcup(\dom(\D_{q^\ast}),\in, R, \Phi, M_i\ra_{i<n}\into \la\bigcup(\dom(\D_q)), \in, R, \Phi, N_i\ra_{i<n}$$ and $q^\dag\in Q\cap\mtcl P_\a$ extends $q^\ast$, then $$\ov q=(F_{q^\dag}, \D_{q^\dag}\cup \D_q\cup\{(M, 0)\,:\,M\in\mtcl N\}\cup \D)$$ is a condition in $\mtcl P_\a$, where $\mtcl N$ is the collection of all  $\Psi_{N^0, N^1}(M)$ such that $M\in \dom(\D_{q^\dag})\cap N^0$, $N^0\in\dom(\D_{q^\ast})$, $N^1\in\dom(\D_q)$, and $\d_{N^0}=\d_{N^1}$, and where $\D$ is the collection of all pairs $(\Psi(M), \g)$ such that there is some $i$ with
\begin{enumerate}
\it $(M, \g)\in M_i\cap\D_{q^\dag}$,
\it $\g\leq \tld\g_i$, and
\it $M_i\in \mathfrak N_{\g+1}$
\end{enumerate}
\end{enumerate}
\end{lemma}

\begin{proof}
 We prove the first part first. For all $i$ let $\tld\g_i$ be obtained from $\g_i$ as in the definition of reflection of $q$ to $Q$. Of course $\tld\g_i=\g_i$ unless $\a=\d$. In any case $\tld\g_i\in Q$ since $Q$ is $\o$--closed. It is straightforward to see that $\tld q=(F_q, \{(N_i, \tld\g_i)\,:\,i<n\})$ is a condition in $\mtcl P_\a$ extending $q_0$. Note also that $F_q$ is in $Q$ since for every $\x<\a$, $X_\x\in Q$, $X_\x\sub Q$, and $\mtcl P_\x$ has the $\al_2$--c.c.\ (and therefore for every $\x<\a$ there is a set $Y_\x\sub H(\k)$, $Y_\x\in Q$, such that for every $v\in H(\k)$ and every $\mtcl P_\x$--condition $r$, if $r\Vdash_\x\check{v}\in \Upsilon(\x)$, then $v\in Y_\x$). Since $Q\elsub H(\t)$ contains all reals, $q_0$, $F_q$, $(\tld\g_i)_{i<n}$ and $\{(N_i\cap Q, \tld\g_i)\,:\,i<n\}$, we may find in $Q$ a reflection of $q$ to $Q$ extending $q_0$:
The existence of such a $q^\ast$ can be expressed by a sentence with parameters in $Q$ and this sentence is true in $H(\t)$ as witnessed by $\tld q$.

We will prove (2) by induction on $\a$. Let $q^\dag\in Q\cap\mtcl P_\a$ extend $q^\ast$ and let $\ov q$ be obtained from $q^\dag$ as in (2). We prove by induction on $\b\leq\a$ that $\ov q$ is a $\mtcl P_\b$--condition. For $\b=0$ this follows immediately from Lemma \ref{iso4}. The present use of partial symmetric systems rather than fully symmetric systems in the sense of \cite{ASP} is needed precisely to make this point of the argument go through. With the stronger notion of symmetry from \cite{ASP} it might not be possible to amalgamate even the restriction to $0$ of $\ov q$ and $q$ into a $\mtcl P_0$--condition.
 The limit case follows from the induction hypothesis. For the successor case, let $\b$ be such that $\b+1<\a$ and suppose $\ov q\av_\b\in\mtcl P_\b$. It suffices to show that if $\b$ is in the support of $q$, then $\ov q\av_\b$ forces that $F_{q^\dag}(\b)$ is \begin{enumerate}
 \it $(N_i[\dot G_\b],\,\Upsilon(\b))$--generic for every $i<n$ such that $\g_i>\b$ and $N_i\in\mathfrak N_{\b+1}$, and
 \it $(\Psi(M)[\dot G_\b],\,\Upsilon(\b))$--generic whenever $M\in\mtcl N_{\b+1}$, $(M, \g)\in\D_{q^\dag}$ for some $\g\geq\b+1$, and there is some $i$ such that $M\in M_i$, $\g\leq\tld\g_i$ and $M_i\in\mathfrak N_{\b+1}$.
\end{enumerate}

We prove the first part only (the proof of the second part is essentially the same).
Note that for every $i$ as in (1), $M_i\in\mathfrak N_{\b+1}$ and $\tld\g_i>\b$. Hence, $\ov q\av_\b$ forces that $F_{q^\dag}(\b)$ is $(M_i[\dot  G_\b],\,\Upsilon(\b))$--generic since $\ov q\av_\b$ extends $q^\dag\av_\b$ and $q^\dag\av_\b$ forces this. Suppose towards a contradiction that there is a $\mtcl P_\b$--condition $t$ extending $\ov q\av_\b$, a $\mtcl P_\b$--name $\dot A\in N_i$ for a maximal antichain of $\Upsilon(\b)$, and some $y\in H(\k)\cap Q$ such that $t\Vdash_\b\check y\in\Upsilon(\b)$ and such that $t$ forces that $\check y$ extends $F_{q^\dag}(\b)$ and that no condition in $\dot A \cap N_i[\dot G_\b]$ is compatible with $y$.  We may assume that $\dot A$ consists of pairs $(r,  \check{x})$, with $r\in\mtcl P_\b$ and $\check{x}$ the canonical $\mtcl P_\b$--name of some $x\in H(\k)^{\V}$. Then $\Psi^{-1}(\dot A)\in M_i$ is a $\mtcl P_\b$--name for a maximal antichain of $\Upsilon(\b)$ by the choice of $\Psi$ since $\Psi^{-1}\restr N_i$ preserves the predicate $\mtcl P_\b$. Note that,  by the $\al_2$--c.c.\ of $\mtcl P_\b$ together with the $\al_2$--c.c.\ of $\Upsilon(\b)$ in $\V[G_\b]$, there is a surjection $\varphi:\o_1\into \dot A$ in $N_i$. Let now $t^\ast$ be a reflection of $t$ to $Q$ extending $q^\dag\av_\b$ (which exists by (1) for $\b$) and let $t'\leq_\b t^\ast$, $t'\in Q$, and $y'\in H(\k)\cap Q$ be such that $t'$ forces that $y'$ is a common extension in $\Upsilon(\b)$ of $y$ and of some $x \in\Psi^{-1}(\dot A)\cap M_i$. These objects exist since $q^\dag\av_\b$ is $(M_i, \mtcl P_\b)$--generic and forces $F_{q^\dag}(\b)$ to be $(M_i[\dot G_\b],\,\Upsilon(\b))$--generic and since $t^\ast$ extends $q^\dag\av_\b$. By extending $t'$ within $Q$ if necessary we may of course assume that $t'$ decides $x$, that $(r, \check{x})\in\Psi^{-1}(\dot A)\cap M_i$ for some  $r$ and that $t'$ extends $r$. There is some $\n<\d_{M_i}$ such that $\Psi^{-1}(\varphi)(\n)=(r, \check{x})$. But then $\varphi(\n)=(\Psi(r), \check{x})\in\dot A\cap N_i$. Now, by induction hypothesis there is a condition $\tld t$ in $\mtcl P_\b$ extending  $t$, $t'$ and  $\Psi(r)$. This is a contradiction since $\tld t$ forces that $x\in\dot A\cap N_i$ is compatible with $y$.
\end{proof}

\begin{lemma}\label{fa}
$\mtcl P_\k$ forces $\MA^{1.5}_{<\kappa}$.
\end{lemma}

\begin{proof}
Let $q$ be a $\mtcl P_\k$--condition, let $\chi < \k$, and let $\dot{\mtcl R}\sub H(\k)$ and $\dot A_i$ (for $i<\chi$) be $\mtcl P_\k$--names such that $q$ forces that $\dot{\mtcl R}$ is an $\aleph_{1.5}$--c.c.\ poset relative to $\V$ and $\dot G_\k$ and that each $\dot A_i$ is a maximal antichain of $\dot{\mtcl R}$. By Lemma \ref{bound}, Corollary \ref{cor5} and Lemma \ref{v-proper} we will be done if we show that there is some condition extending $q$ and forcing that there is a filter on $\dot{\mtcl R}$ intersecting all $\dot A_i$. Let $X$ be a subset of $\k$ coding the $\mtcl P_\k$--name $\dot{\mtcl R}$ via our fixed translating function $\p$.

Now, using the fact that $\vec X$ is a $\diamondsuit(\{\a<\k\,:\,\cf(\a)\geq\o_2 \})$--sequence, we may fix an elementary substructure $Q$ of some large enough $H(\t^\ast)$ containing $\Phi$, $q$, $\mtcl P_\k$, $\dot{\mtcl R}$, $(\dot A_i)_{i\in\chi}$, $\vec X$, $X$ and a well--order $W$ of $H(\t)$ for some strong limit $\t>\kappa$, closed under $\o$--sequences, and such that $\d= Q\cap\k$ is an ordinal such that $X_\d=X\cap\d$ (since $\m^{\al_0} < \k$ for all $\m < \k$, the set of $\d\in\k$ for which there is a $Q$ as above contains a $\l$--club for every regular cardinal $\l<\k$, $\l\geq\o_1$).  Let $\dot{\mtcl R}_0$ be the $\mtcl P_\d$--name coded by $X_\d$. By taking $\d$ from within a suitable club we may further assume that $q$ forces in $\mtcl P_\k$ for all $\x$, $\x'$ in $\d$ that if $\p_{X_\d}(\x)$ and $\p_{X_\d}(\x')$ are compatible conditions in $\dot{\mtcl R}$, then there is an ordinal below $\d$ coding a common extension in $\dot{\mtcl R}$ of $\p_{X_\d}(\x)$ and $\p_{X_\d}(\x')$ (since this is true for a club of $\k$ in the extension which, by the $\k$--c.c.\ of $\mtcl P_\k$, includes a club of $\k$ in $\V$). Using this, we may assume as well that $\d$ is taken so that $q$ forces in $\mtcl P_\k$, for every  $\mtcl P_\d$--name for a maximal antichain of $\dot{\mtcl R}_0$, that $\dot A$ is a maximal antichain of $\dot{\mtcl R}$. This is true since every $\dot A$ as above can be taken to be in $H(\k)$ and since $\Phi:\k\into H(\k)$ is a surjection.

The following claim follows from the closure of $Q$ under $\o$--sequences together with the above choice of $\d$.

\begin{claim} $q$ forces in $\mtcl P_\d$ that $\dot{\mtcl R}_0$ is an $\aleph_{1.5}$--c.c.\ poset relative to $\V$ and $\dot G_\d$.
\end{claim}

\begin{proof}
Let $q'\leq_\d q$ and let $\{N^\ast_0,\ldots, N^\ast_{n-1}\}$ be a set of countable elementary substructures of $\la H(\t), \in, W\ra$ containing everything relevant (this includes a name $\dot D$ for a club witnessing the $\al_{1.5}$--c.c.\ of $\dot{\mtcl R}$ relative to $\V$ and $\dot G_\k$) and such that $(N_i, \d)\in\D_{q'}$ for all $i$, where $N_i:=N_i^\ast\cap H(\k)$. Suppose towards a contradiction that there is some $y\in Q$ such that $q'\Vdash_\d\check{y}\in\dot{\mtcl R}_0$ and such that $y\in N_0$, where $\d_{N_0}=\min\{\d_{N_0},\ldots,\d_{N_{n-1}}\}$, and that $q'$ forces that there is no extension of $y$ which is $(N_i[\dot G_\d],\,\dot{\mtcl R}_0)$--generic for all $i<n$.

By the proof of Lemma \ref{refl} (1) together with the proof of Subclaim \ref{copias} we may find a reflection $q^\ast$ of $q'$ to $Q$ so that, in addition, if $\Psi$ is the isomorphism witnessing that $q^\ast$ is a reflection of $q'$ to $Q$ and $M_i=\Psi^{-1}`` N_i$ for all $i$, then there are $M_i^\ast\in Q$ (for $i<n$) such that each $M^\ast_i$ is a countable elementary substructure of $\la H(\t), \in, W\ra$ containing everything relevant and such that $M_i=M^\ast_i\cap H(\k)$ and $(M^\ast_i, \in, W)$ is isomorphic to $(N^\ast_i, \in, W)$ by an isomorphism fixing $\dot D$ and $\dot R$ whose restriction to $M_i$ is $\Psi\restr M_i$. We may also assume that $\Psi\restr M_i$ preserves $\mtcl P_\b$ for every $i$ and every $\b\in N_i\cap Q$. Clearly we may assume that $q^\ast$ also forces $y\in\dot{\mtcl R}$.  Let $q^\dag\in Q$ be an extension of $q^\ast$ and let $y'\in H(\k)\cap Q$ be such that $q^\dag$ forces that $y'$ is a condition in $\dot{\mtcl R}$ extending $y$ and $(M_i[\dot G_\k],\,\dot{\mtcl R})$--generic for all $i$.  By Lemma \ref{refl} (2), $q^\dag$ and $q'$ can be amalgamated into a $\mtcl P_\d$ condition $\ov q$. Now we argue that $\ov q$ forces that $y'$ is $(N_i[\dot G_\d],\,\dot{\mtcl R}_0)$--generic for all $i$.

For this, suppose towards a contradiction that there is a $\mtcl P_\d$--condition $\ov q'$ extending $\ov q$, some $i$, and some $\dot A\in N_i$ such that $\dot A$ is a $\mtcl P_\d$--name such that $\ov q'$ forces that $\dot A$ is a maximal antichain of $\dot{\mtcl R}_0$ and that there is no condition in $\dot A\cap N_i$ compatible with $y'$. Note that, by our choice of $\d$, $\dot A$ is a $\mtcl P_\k$--name such that $\ov q'$ forces in $\mtcl P_\k$ that $\dot A$ is a maximal antichain of $\dot{\mtcl R}$.  We may therefore assume that $\dot A$ is actually a $\mtcl P_\k$--name for a maximal antichain of $\dot{\mtcl R}$. Let $q^{\ast\ast}$ be a reflection of $\ov q'$ to $Q$.  The rest of the argument is again as in the proof of Lemma \ref{refl} (2): We find an extension $q^{\ddag}\in Q$ of $q^{\ast\ast}$ together with some $y''$ and some $x\in M_i$ such that $q^{\ddag}$ forces that $\check x\in\Psi^{-1}(\dot A)$ and that $y''$ is a common extension in $\dot{\mtcl R}$ of  $x$ and $y'$. We then find by (the proof of) Lemma \ref{refl} (2) an extension of $q^{\ddag}$ and $\ov q'$ forcing $\check x\in \dot A\cap N_i$, which is a contradiction. Exactly as in the proof of  Lemma \ref{refl} (2), the argument uses the fact that $\dot{\mtcl R}$ has the $\al_2$--c.c.\ and so, since $\mtcl P_\d$ has also the $\al_2$--c.c., there is an enumeration of $\dot A$ in $N_i$ in length $\o_1$.
\end{proof}

It follows from the above claim that $q$ forces $\Upsilon(\d)=\dot{\mtcl R}_0$. Finally, we may extend $q$ to a condition $q'$ such that $\d\in \supp(q')$. Then, by Lemma \ref{genh}, $q'$ forces that there is a filter $H$ on $\dot{\mtcl R}_0$ meeting all $\dot A_i$, and of course $H$ generates a filter on $\dot{\mtcl R}$.
\end{proof}

\begin{lemma}\label{kappa}
$\mtcl P_{\k}$ forces $2^{\al_0}=\k$.
\end{lemma}

\begin{proof}
$\V^{\mtcl P_\k}\models 2^{\al_0}\geq\k$ follows for example from the fact that $\mtcl P_\k$ forces $\MA^{1.5}_{<\kappa}$. $\V^{\mtcl P_\k}\models 2^{\al_0}\leq\k$ follows from Lemma \ref{bound}.
\end{proof}

Lemma \ref{kappa} finishes the proof of Theorem \ref{mainthm}.

\end{document}